\documentclass[12pt]{amsart}

\usepackage[utf8]{inputenc}
\usepackage{array}
\usepackage{enumitem}
\usepackage[T1]{fontenc}
\usepackage{tgpagella}
\usepackage{newpxmath}

\usepackage[nocompress]{cite}
\usepackage{hyperref}

\newcommand{\bbZ}{\mathbb{Z}}
\newcommand{\bbC}{\mathbb{C}}
\newcommand{\bbQ}{\mathbb{Q}}
\newcommand{\bbR}{\mathbb{R}}
\newcommand{\Rninf}{\mathbb{R} \cup \{-\infty\}}
\newcommand{\bbD}{\mathbb{D}}

\newcommand{\bfze}{\mathbf{0}}

\NewDocumentCommand{\cI}{o}{%
    \IfNoValueTF{#1}{\mathcal{I}}{\mathcal{I}\big( #1 \big)}%
}
\NewDocumentCommand{\cJ}{o}{%
    \IfNoValueTF{#1}{\mathcal{J}}{\mathcal{J}\big( #1 \big)}%
}

\newcommand{\cO}{\mathcal{O}}
\newcommand{\cF}{\mathcal{F}}

\newcommand{\inp}[2]{\langle #1 , #2 \rangle}

\newcommand{\mfa}{\mathfrak{a}}

\newcommand{\reg}{\mathrm{reg}}
\newcommand{\loc}{\mathrm{loc}}

\DeclareMathOperator{\interior}{int}
\DeclareMathOperator{\Supp}{Supp}
\DeclareMathOperator{\divisor}{div}
\DeclareMathOperator{\dom}{dom}
\DeclareMathOperator{\sing}{sing}
\DeclareMathOperator{\ord}{ord}
\DeclareMathOperator{\Hom}{Hom}
\DeclareMathOperator{\Log}{Log}
\DeclareMathOperator{\rec}{rec}
\DeclareMathOperator{\orb}{orb}

\DeclareMathOperator{\conv}{Conv}
\DeclareMathOperator{\cone}{Cone}

\newcommand{\cM}{\mathcal{M}}
\DeclareMathOperator{\Exc}{Exc}

\theoremstyle{plain}

\newtheorem{theorem}{Theorem}[section]
\newtheorem{lemma}[theorem]{Lemma}
\newtheorem{corollary}[theorem]{Corollary}
\newtheorem{proposition}[theorem]{Proposition}
\newtheorem{definition}[theorem]{Definition}
\newtheorem{example}[theorem]{\textnormal{\textbf{Example}}}

\theoremstyle{remark}
\newtheorem{remark}[theorem]{Remark}

\title[Analytic multiplier ideals on a toric complex space]{Analytic multiplier ideals of toric plurisubharmonic functions on a toric complex space}

\author[Hoseob Seo]{Hoseob Seo}
\address{Research Institute of Mathematics, Seoul National University, 1 Gwanak-ro, Gwanak-gu, Seoul, 08826, Republic of Korea }
\email{hoseobseo@snu.ac.kr; hsseo.math@gmail.com}

\subjclass[2020]{Primary 14F18, 32U05; Secondary 14M25}
\keywords{Analytic multiplier ideals, toric plurisubharmonic functions, toric varieties, Newton convex bodies, monomial ideals}

\begin{document}

\begin{abstract}
In this paper, we study analytic multiplier ideal sheaves of toric plurisubharmonic functions on normal affine toric complex analytic spaces equipped with torus-invariant rational divisors, assuming that a positive integer multiple of the log canonical divisor is Cartier. We describe these sheaves in terms of Newton convex bodies and show that they are globally generated by finitely many meromorphic monomials. Our results generalize the description of An and the formulas of Howald and Blickle for monomial ideals, and recover the descriptions of Rashkovskii and Guenancia on affine space. We also extend this description to weights obtained by subtracting positive rational multiples of monomial ideal weights.
\end{abstract}

\maketitle

\section{Introduction}

Multiplier ideal sheaves play important roles in complex and algebraic geometry, with applications to vanishing theorems and the $L^2$ extension of holomorphic sections (see, e.g., \cite{demailly2012complex,lazarsfeld2004positivity2}). They are used to study singularities of divisors and ideal sheaves, as well as those of plurisubharmonic (psh, for short) functions. In particular, log canonical thresholds and jumping numbers give numerical information on these singularities through the changes of multiplier ideals.
Variants of multiplier ideals and related ideal sheaves have been developed to address different questions about singularities, adjunction, and extension, and have been intensively studied in recent years (see, e.g., \cite{defernex2009singularities,fujino2010nonlc,fujino2011supplements,li2021briancon,chan2023adjoint,kimseo2023extension,li2024nadel}).

Explicit computations of multiplier ideals are important for turning vanishing and extension theorems into concrete geometric statements. For example, on a smooth projective variety, under suitable positivity hypotheses, Nadel vanishing implies that global sections of an adjoint line bundle restrict surjectively to the subscheme defined by a multiplier ideal. An explicit description of this ideal identifies the subscheme and the conditions imposed on sections, and can therefore be used to construct sections with prescribed values or jets. Computing or controlling multiplier ideals is thus a key ingredient in arguments for global generation and separation of points.

In the toric setting, these ideals admit explicit descriptions in terms of convex sets. For monomial ideals on $\bbC^n$, Howald \cite{howald2001multiplier} gave a description using Newton polyhedra.
In \cite{blickle2004multiplier}, Blickle extended this result to normal affine toric pairs. Let $X_\sigma$ be the normal affine toric variety associated to a strongly convex rational polyhedral cone $\sigma$, with character lattice $M$. Let $D=\sum_\rho a_\rho D_\rho$ be a torus-invariant Weil $\bbQ$-divisor such that $K_{X_\sigma}+D$ is $\bbQ$-Cartier. Choose $u_D\in M_\bbQ$ such that $\inp{u_D}{v_\rho}=1-a_\rho$
for every primitive ray generator $v_\rho$ of $\sigma$.

\begin{theorem}[Blickle {\cite[Theorem~1]{blickle2004multiplier}}]
    Let $\mfa\subset\cO_{X_\sigma}$ be a monomial ideal sheaf. Then
    $$
        \cJ[(X_\sigma,D);\mfa^c] =
        \bigl\langle \chi^m \,\bigm|\, m+u_D\in\interior\bigl(c\,\operatorname{Newt}(\mfa)\bigr) \bigr\rangle,
    $$
    for all $c>0$, where $\operatorname{Newt}(\mfa)$ is the Newton polyhedron of $\mfa$.
\end{theorem}

When $X_\sigma=\bbC^n$ and $D=0$, this is Howald's formula. For psh functions, Newton polyhedra are replaced by Newton convex bodies. A psh function $\varphi$ on a polydisk centered at the origin in $\bbC^n$ is said to be toric if $\varphi(z_1,\ldots,z_n)$ depends only on $|z_1|,\ldots,|z_n|$.

\begin{theorem}[Rashkovskii {\cite[Proposition~3.1]{rashkovskii2013multi}}, Guenancia {\cite[Theorem~1.13]{guenancia2012toric}}]
    Let $\varphi$ be a toric psh function on a polydisk centered at the origin in $\bbC^n$. Then the stalk $\cI (\varphi)_0$ is a monomial ideal. A monomial $z^\alpha$ with $\alpha \in \bbZ_{\ge 0}^n$ belongs to $\cI(\varphi)_0$ if and only if $\alpha + \mathbf{1} \in \interior P(\varphi)$, where $\mathbf{1} = (1,\ldots,1)$ and $P(\varphi)$ is the Newton convex body of $\varphi$.
\end{theorem}

Explicit toric descriptions of multiplier ideals have also been used to study jumping numbers and their cluster points \cite{ein2004jumping,kimseo2020jumping,seo2021cluster}, connections with Bernstein--Sato polynomials \cite{budur2006bernstein,hsiao2018bernstein}, and equisingular approximation of toric psh functions \cite{anseo2023equisingular}.

In \cite{an2021Thesis}, An obtained an analogous characterization on normal $\bbQ$-Gorenstein toric varieties. Our main interest is to extend these characterizations to toric psh functions on pairs $(X_\sigma,D)$ as above, with $X_\sigma$ viewed as a complex analytic space.
We denote by $\cI[(X_\sigma,D);\varphi]$ the multiplier ideal sheaf defined through adapted volume forms, and by $P(\varphi)$ the Newton convex body defined in Section~\ref{sec:newton_convex_bodies_and_invariant_analytic_sheaves}.

\begin{theorem}[Theorem~\ref{thm:toric_multiplier_ideal}]
    Let $\varphi$ be a toric psh function on $X_\sigma$ in the sense of Definition~\ref{def:toric_psh}.
    Then $\cI[(X_\sigma,D);\varphi]$ is a coherent $T_N$-invariant $\cO_{X_\sigma}$-submodule of $\cM_{X_\sigma}$ globally generated by finitely many meromorphic monomials.
    For every $m\in M$, the monomial $\chi^m$ is a global section of this sheaf if and only if $m+u_D\in\interior P(\varphi)$.
\end{theorem}

When $X_\sigma=\bbC^n$ and $D=0$, we have $u_D=\mathbf{1}$, and the main theorem specializes to the description of Rashkovskii and Guenancia for toric psh functions defined on $\bbC^n$.
If $\varphi_\mfa$ is a toric psh function associated to a nonzero monomial ideal $\mfa$ and $c>0$, then $\interior P(c\varphi_\mfa)=\interior (c\operatorname{Newt}(\mfa))$. Together with Theorem~\ref{thm:analytic_mis_extend_algebraic_mis}, the main theorem recovers Blickle's theorem.
When $\sigma$ is full-dimensional, $P(\varphi)$ coincides with the Newton convex body defined by An through a toric resolution (Proposition~\ref{prop:comparison_with_An}), so the main theorem extends An's description to pairs and to cones that are not full-dimensional.

We also obtain a monomial criterion for weights $\varphi-c\varphi_\mfa$, where $c>0$ is rational (Corollary~\ref{cor:toric_negative_ideal_weight_criterion}).
Such weights are in general not psh, so the main theorem does not apply to them directly. They arise naturally: for instance, the maximal non-lc ideal sheaf $\cJ'(X,D)$ with $D$ effective, introduced by Fujino--Schwede--Takagi, is computed by $\cI[(X,D);-\epsilon\varphi_{\mathfrak b}]$ with $\mathfrak b=\cI[(X,D);0]$ and sufficiently small $\epsilon>0$ (\cite[Lemma~7.5]{fujino2011supplements}; see Proposition~\ref{prop:analytic_characterization_of_maximal_non-lc_ideal}).

\begin{theorem}[Corollary~\ref{cor:toric_negative_ideal_weight_criterion}]
    Let $\varphi$ be a toric psh function on $X_\sigma$. Let $\mfa \subset \cO_{X_\sigma}$ be a nonzero coherent $T_N$-invariant ideal sheaf, and let $c$ be a positive rational number.
    Write
    $$
        \mfa = \chi^{a_1}\cO_{X_\sigma}+\cdots+\chi^{a_s}\cO_{X_\sigma}.
    $$
    Then $\cI[(X_\sigma,D);\varphi-c\varphi_\mfa]$ is a coherent $T_N$-invariant $\cO_{X_\sigma}$-submodule of $\cM_{X_\sigma}$ globally generated by finitely many meromorphic monomials.
    For every $m\in M$, the monomial $\chi^m$ is a global section of this sheaf if and only if $m+u_D+ca_j\in\interior P(\varphi)$ for every $j=1,\ldots,s$.
\end{theorem}

As an application, we recover the known facts that a normal toric variety with its reduced toric boundary is lc and that a $\bbQ$-Gorenstein toric variety is klt (cf. \cite[Proposition~3.7]{kollar1997singularities}).

There are two crucial ingredients for the proof. First, we establish transformation formulas, coherence, and strong openness for $\cI[(X,D);\varphi]$ with quasi-psh weights, locally the sum of a psh function and a smooth function, when the boundary may have negative coefficients. We also show that these sheaves are ideals when $D$ is effective and $\varphi$ is merely locally bounded above.
Second, we show that a coherent $K_N$-invariant analytic $\cO_{X_\sigma}$-submodule of $\cM_{X_\sigma}$ is $T_N$-invariant and globally generated by finitely many meromorphic monomials (Theorem~\ref{thm:compact_torus_invariant_imply_torus_invariant} and Corollary~\ref{cor:global_meromorphic_monomial_generation}).
To determine which meromorphic monomials are sections, we use convex analysis, particularly the relation between Legendre transforms and recession functions, to express local integrability in terms of Newton convex bodies.

The paper is organized as follows. In Section~\ref{sec:preliminaries}, we recall preliminaries on toric varieties and psh functions. In Section~\ref{sec:multiplier_ideal_sheaves_on_normal_complex_analytic_spaces}, we discuss analytic multiplier ideal sheaves on normal complex analytic spaces. In Section~\ref{sec:legendre_transforms_and_recession_functions}, we give the results on Legendre transforms and recession functions needed for the integrability criterion. In Section~\ref{sec:newton_convex_bodies_and_invariant_analytic_sheaves}, we study Newton convex bodies and invariant analytic sheaves. Finally, in Section~\ref{sec:multiplier_ideals_on_a_toric_complex_space}, we prove the main theorem and give its applications.

\medskip

\noindent \textbf{Acknowledgements.} This work was supported by the National Research Foundation of Korea (NRF) grant funded by the Korean government (MSIT) (RS-2025-00556878).

\section{Preliminaries} \label{sec:preliminaries}
We denote by $\bbZ_{\ge 0}$ (resp. $\bbZ_{> 0}$) the set of nonnegative (resp. positive) integers. The sets $\bbR_{\ge 0}$ and $\bbR_{> 0}$ are defined similarly.
\subsection{Toric varieties}
Let $N$ be a free $\bbZ$-module of rank $r$ and denote by $N_\bbR$ the $\bbR$-vector space obtained by extending scalars to $\bbR$. Let $M$ be the dual $\bbZ$-module of $N$ and $M_\bbR$ the corresponding $\bbR$-vector space. We refer to \cite{oda1988convex, fulton1993introduction, cox2011toric} for details on toric varieties.
\begin{definition}
    Let $\sigma$ be a subset of $N_{\bbR} = N \otimes_\bbZ \bbR$. It is said to be a {\bf convex polyhedral cone} if there exist a finite number of vectors $n_1,\ldots,n_s \in N_\bbR$ such that
    $$
        \sigma = \cone(n_1,\ldots, n_s) = \{ a_1n_1 + \cdots + a_sn_s \, :\, \text{$a_i \ge 0$ for all $i$}\}.
    $$
    For a convex polyhedral cone $\sigma \subset N_\bbR$,
    \begin{enumerate}
        \item $\dim\sigma$ is the {\bf dimension} of the smallest $\bbR$-subspace containing $\sigma$.
        \item $\sigma$ is said to be {\bf rational} if a set of generators of $\sigma$ can be chosen from $N$.
        \item $\sigma$ is said to be {\bf strongly convex} if $\sigma \cap (-\sigma) = \{ \bfze\}$.
    \end{enumerate}
\end{definition}

Since $M_\bbR$ is the dual of $N_\bbR$ as an $\bbR$-module, we have an $\bbR$-bilinear pairing 
$$
\inp{\,}{}: M_\bbR \times N_\bbR \to \bbR,
$$
which extends the canonical pairing between $M$ and $N$. 

\begin{definition}
    Let $\sigma \subset N_\bbR$ be a convex cone. The {\bf dual cone} of $\sigma$, denoted by $\sigma^\vee$, is defined by
    $$
        \sigma^\vee = \{ m \in M_\bbR \,:\, \text{$\inp{m}{n} \ge 0$ for all $n \in \sigma$} \}.
    $$
\end{definition}

If $\sigma$ is a rational polyhedral cone in $N_\bbR$, it is well known that $\sigma^\vee \cap M$ is a finitely generated semigroup, say
$$
\sigma^\vee \cap M = \bbZ_{\ge 0} m_1 + \cdots + \bbZ_{\ge 0} m_p.
$$
Let $\bbC[\sigma^\vee \cap M]$ be the $\bbC$-algebra generated by $\chi^{m_1},\ldots, \chi^{m_p}$ with exponential law $\chi^{m + m'} = \chi^{m} \cdot \chi^{m'}$.
Denote by $X_\sigma$ the maximal spectrum of $\bbC[\sigma^\vee \cap M]$, that is, 
$$
 X_\sigma = \{ \mathfrak{m} \,:\, \text{$\mathfrak{m}$ is a maximal ideal of $\bbC[\sigma^\vee \cap M]$}\}.
$$
There is a canonical one-to-one correspondence between $X_\sigma$ and the set of all $\bbC$-algebra homomorphisms from $\bbC[\sigma^\vee \cap M]$ to $\bbC$. Moreover, the surjective homomorphism from $\bbC[z_1,\ldots, z_p]$ to $\bbC[\sigma^\vee \cap M]$, determined by $z_i \mapsto \chi^{m_i}$, allows us to consider $X_\sigma$ as a subset of $\bbC^p$. Thus $X_\sigma$ has a complex analytic structure naturally induced from $\bbC^p$. Its complex analytic structure is, of course, independent of the choice of generators of $\sigma^\vee \cap M$.

\begin{definition}
    The {\bf affine toric complex analytic space} $X_\sigma$ associated to a strongly convex rational polyhedral cone $\sigma$ is the maximal spectrum of $\bbC[\sigma^\vee \cap M]$, endowed with the induced complex analytic structure.
\end{definition}

\begin{definition}
    Let $\Delta$ be a nonempty collection of strongly convex rational polyhedral cones in $N_\bbR$. We call $\Delta$ a {\bf fan} if it satisfies
    \begin{enumerate}
        \item Every face of a cone $\sigma \in \Delta$ is contained in $\Delta$.
        \item For any two cones $\sigma,\sigma' \in \Delta$, their intersection $\sigma \cap \sigma'$ is a face of $\sigma$ and $\sigma'$ simultaneously.
    \end{enumerate}
\end{definition}

For two cones $\sigma$ and $\sigma'$ in a fan $\Delta$ with a common face $\tau = \sigma \cap \sigma'$, the affine toric complex analytic space $X_\tau$ can be regarded as an open subset of both $X_\sigma$ and $X_{\sigma'}$. We can glue $X_\sigma$ and $X_{\sigma'}$ together along $X_\tau$. In this way, we can glue all affine toric complex analytic spaces obtained from the cones in $\Delta$ and the resulting space with complex analytic structure is called a {\bf toric complex analytic space}, denoted by $X_\Delta$.

Note that the affine toric complex analytic space $X_{\{\bfze\}}$ corresponding to the smallest cone $\{\bfze\}$ in $\Delta$ is the algebraic torus $T_N = \Hom_\bbZ (M, \bbC^*)$. The torus $T_N$ acts on $X_\Delta$ by extending its multiplication action on the dense open torus $T_N \subset X_\Delta$.  For $\sigma \in \Delta$, denote by $\orb(\sigma)$ the corresponding $T_N$-orbit in $X_\Delta$.

\begin{proposition}[{\cite[Proposition~1.6]{oda1988convex}}]
    The cones of a fan $\Delta$ are in one-to-one correspondence with $T_N$-orbits in $X_\Delta$.  Moreover, the following hold:
    \begin{enumerate}
        \item $\orb(\{\bfze\}) = X_{\{\bfze\}} = T_N$.
        \item For $\sigma \in \Delta$, the complex analytic dimension of $\orb (\sigma)$ coincides with the codimension $r - \dim \sigma$ of $\sigma$ in $N_\bbR$.
        \item For $\sigma,\tau \in \Delta$, $\tau$ is a face of $\sigma$ if and only if $\orb(\sigma)$ is contained in the closure of $\orb(\tau)$.
        \item For $\sigma \in \Delta$, $\orb(\sigma)$ is the unique closed $T_N$-orbit in $X_\sigma$ and we have
        $$
            X_\sigma = \coprod_{\tau < \sigma} \orb(\tau).
        $$
    \end{enumerate}
\end{proposition}

The orbit--cone correspondence also describes the $T_N$-invariant prime divisors on $X_\Delta$. We recall this description and the Cartier criterion that will be used later in Section~\ref{sec:multiplier_ideals_on_a_toric_complex_space}.

Denote by $\Delta(1)$ the set of rays of $\Delta$ and by $v_\rho$ the primitive generator of $\rho \in \Delta(1)$. The divisors
$$
    D_\rho := \overline{\orb(\rho)}, \quad \rho \in \Delta(1),
$$
are precisely the $T_N$-invariant prime divisors on $X_\Delta$. Thus, a $T_N$-invariant Weil $\bbQ$-divisor has the form
$$
    D = \sum_{\rho \in \Delta(1)} a_\rho D_\rho, \quad a_\rho \in \bbQ.
$$
For every $m \in M$, the character $\chi^m$, viewed as a meromorphic function on $X_\Delta$, satisfies
$$
    \divisor(\chi^m) = \sum_{\rho \in \Delta(1)} \inp{m}{v_\rho} D_\rho.
$$

The Cartier property of a $T_N$-invariant divisor can be expressed in terms of the primitive ray generators.

\begin{theorem}[{\cite[Theorem~4.2.8]{cox2011toric}}] \label{thm:toric_cartier_criterion}
    Let $D = \sum_{\rho \in \Delta(1)} a_\rho D_\rho$ be a $T_N$-invariant Weil divisor on $X_\Delta$. Then $D$ is Cartier if and only if, for every $\sigma \in \Delta$, there exists $m_\sigma \in M$ such that
    $$
        \inp{m_\sigma}{v_\rho} = a_\rho \quad \text{for every $\rho \in \sigma(1)$,}
    $$
    where $\sigma(1)$ denotes the set of rays of $\sigma$. In this case,
    $$
        D|_{X_\sigma} = \divisor (\chi^{m_\sigma}).
    $$
\end{theorem}

As a consequence of the above theorem, a $T_N$-invariant Weil $\bbQ$-divisor $D$ is $\bbQ$-Cartier if and only if the same condition holds with $m_\sigma \in M_\bbQ$.

We choose the canonical divisor
$$
    K_{X_\Delta} = -\sum_{\rho \in \Delta(1)} D_\rho,
$$
which is the divisor of the meromorphic top-degree form
$$
    \frac{d\chi^{m_1}}{\chi^{m_1}} \wedge \cdots \wedge \frac{d\chi^{m_r}}{\chi^{m_r}}
$$
on the open torus $T_N \subset X_\Delta$, where $m_1,\ldots,m_r$ is a basis of $M$.

For a $T_N$-invariant Weil $\bbQ$-divisor $D = \sum_{\rho \in \sigma(1)} a_\rho D_\rho$ on an affine toric variety $X_\sigma$, the divisor $K_{X_\sigma}+D$ is $\bbQ$-Cartier if and only if there exists $u_D \in M_\bbQ$ such that
$$
    \inp{u_D}{v_\rho} = 1-a_\rho \quad \text{for every $\rho \in \sigma(1)$.}
$$
In this case, for any positive integer $q$ such that $qu_D \in M$,
$$
    q(K_{X_\sigma} + D) = \divisor (\chi^{-qu_D}).
$$

\subsection{Plurisubharmonic functions}

\begin{definition}
    Let $\Omega$ be an open subset of $\bbC^n$ and let $\varphi:\Omega \to \Rninf$ be a function which is not identically $-\infty$ on any connected component of $\Omega$. Then $\varphi$ is said to be {\bf plurisubharmonic} ({\bf psh} for short) if 
    \begin{enumerate}
        \item $\varphi$ is upper semicontinuous on $\Omega$, and
        \item for every line $\ell$ in $\bbC^n$ such that $\ell \cap \Omega \neq \emptyset$, the restriction of $\varphi$ to $\ell \cap \Omega$ is subharmonic or identically $-\infty$.
    \end{enumerate}
\end{definition}

Plurisubharmonicity is invariant under holomorphic coordinate changes. This notion naturally extends to functions on complex manifolds. For complex analytic spaces, there are several equivalent ways to define plurisubharmonicity. Let us adopt the following definition by Narasimhan \cite{narasimhan1962levi}.

\begin{definition}
    Let $V$ be an open subset of a complex analytic space $X$ and let $\varphi: V \to \Rninf$ be a function which is not identically $-\infty$ on any connected component of $V$. The function $\varphi$ is said to be {\bf plurisubharmonic} ({\bf psh} for short) if, for each $x \in V$, there exist an open neighborhood $U_x$ of $x$ in $V$, an open subset $W$ of $\bbC^p$, a closed embedding $i: U_x \hookrightarrow W$ and a psh function $\Phi$ on $W$ such that $\varphi|_{U_x} = \Phi \circ i$.
\end{definition}

\begin{theorem}[\cite{fornaess1980levi}] \label{thm:psh_iff_subharmonic_along_discs}
    Let $X$ be a complex analytic space and let $\varphi:X \to \bbR \cup \{ -\infty \}$ be an upper semicontinuous function on $X$ which is not identically $-\infty$ on any connected component of $X$. Then $\varphi$ is psh if and only if, for every holomorphic map $f: \bbD \to X$ from the unit disk $\bbD$ to $X$, the function $\varphi \circ f$ is either subharmonic on $\bbD$ or identically $-\infty$. 
\end{theorem}

\begin{theorem}[\cite{grauert1956plurisubharmonische}, {\cite[Theorem~1.7]{demailly1985mdme}}] \label{thm:extension_of_psh_across_analytic_subset}
    Let $X$ be a locally irreducible complex analytic space and $Y$ a nowhere dense analytic subset of $X$. If $\varphi$ is a psh function on $X \backslash Y$ and $\varphi$ is locally bounded above near $Y$, then $\varphi$ uniquely extends to a psh function on $X$.
\end{theorem}

In the above theorem, the extension $\tilde \varphi$ of $\varphi$ is given by
$$
\tilde \varphi (z) = \limsup_{\substack{w \to z \\ w \in X \backslash Y}} \varphi(w)
$$
for all $z \in X$. 

If a nonzero ideal sheaf $\mfa$ of $\cO_X$ is generated by $f_1,\ldots,f_s$ locally, we can associate to $\mfa$ a quasi-psh function locally defined by
$$
    \varphi_\mfa = \frac{1}{2}\log\bigl(|f_1|^2+\cdots+|f_s|^2\bigr) + \mathrm{O}(1).
$$
Changing the local generators changes this representative only by a locally bounded function. Thus $\varphi_\mfa$ is well-defined up to locally bounded functions, although it is only locally defined. By abuse of notation, we will use $\varphi_\mfa$ as if it were a quasi-psh function. For an effective $\bbQ$-Cartier divisor $B$, the associated quasi-psh function $\varphi_B$ is defined similarly.

\begin{definition} \label{def:toric_psh}
    Let $\sigma$ be a strongly convex rational polyhedral cone in $N_\bbR$. Denote by $K_N$ the maximal compact torus in $T_N$. A psh function $\varphi$ on $X_\sigma$ is said to be {\bf toric} if
    $$
        \varphi(k \cdot x) = \varphi(x)
    $$
    for all $k \in K_N$ and $x \in X_\sigma$.
\end{definition}

\section{Multiplier ideal sheaves on normal complex analytic spaces} \label{sec:multiplier_ideal_sheaves_on_normal_complex_analytic_spaces}

In this section, we recall the definition of the analytic multiplier ideal sheaf attached to a pair $(X,D)$ and develop its basic properties, in particular transformation formulas under modifications, coherence, and the strong openness property. Throughout this section, $(X, D)$ denotes a pair consisting of a normal complex analytic space $X$ of dimension $n$ and a Weil $\bbQ$-divisor $D$ on $X$ such that $K_X + D$ is $\bbQ$-Cartier. Denote by $\omega_X$ the dualizing sheaf of $X$.

A canonical divisor $K_X$ need not exist globally, but it can always be chosen locally. Taking a positive integer $q$ such that $q(K_X + D)$ is Cartier and writing $\omega_X^{[q]} := \big( \omega_X^{\otimes q} \big)^{**}$, we have
$$
\omega_X^{[q]}(qD) := \big( \omega_X^{[q]} \otimes \cO_X(qD) \big)^{**} \simeq \cO_X(q(K_X + D))
$$
as locally free sheaves of rank $1$. On a small open subset $V$ of $X$, there exists a local generator for $\omega_X^{[q]}(qD)$, say $\beta$. Following Eyssidieux--Guedj--Zeriahi \cite{eyssidieux2009singular}, we define the {\bf adapted volume form} $dV_{\beta}$ associated to $\beta$ by
$$
dV_\beta = \Big( \sqrt{-1}^{qn^2} \beta \wedge \bar \beta \Big)^{\frac{1}{q}}.
$$

Throughout this section, all measurable functions are understood to take values in $\bbR \cup \{ \pm \infty \}$ unless otherwise specified.

\begin{definition} \label{def:multiplier_ideal}
    Let $\varphi: X \to \overline \bbR$ be a measurable function that is locally integrable on $X^\reg$ with respect to any smooth volume form. The {\bf multiplier ideal sheaf} $\cI[(X,D);\varphi]$ of $\varphi$ for the pair $(X,D)$ is an $\cO_X$-submodule of $\cM_X$ defined stalk-wise by
    $$
    \cI [(X,D);\varphi]_x = \left\{ f \in \cM_{X,x} \,:\, \text{$\displaystyle \int_{U \cap X^\reg} |f|^2e^{-2\varphi} dV_\beta < \infty$ for a small neighborhood $U$ of $x$} \right\}.
    $$
\end{definition}

While the value of the integral in the definition of $\cI[(X,D);\varphi]$ depends on the choice of $q$ and of the local generator $\beta$, its finiteness does not. Therefore, $\cI[(X,D);\varphi]$ is well-defined. 

By definition, the multiplier ideal sheaf $\cI[(X,D);\varphi]$ is merely an $\cO_X$-submodule of $\cM_X$ and need not be an ideal sheaf of $\cO_X$ even if $\varphi$ is quasi-psh. When $D$ is not effective, the sections of $\cI[(X,D);\varphi]$ are allowed to have poles along the negative part of $D$.

This definition also appears in \cite[Definition~2.2]{li2022question}. Among the results of this section, the statements corresponding to Proposition~\ref{prop:effective_imply_ideal} and Proposition~\ref{prop:transformation_formula} are recorded there for quasi-psh functions (see \cite[Remark~2.3 and Proposition~2.4]{li2022question}). We include complete proofs, and prove the former for measurable functions that are merely locally bounded above.

Recall that, if $X$ is a normal algebraic variety and $D$ is a Weil $\bbQ$-divisor on $X$ such that $K_X + D$ is $\bbQ$-Cartier, the (algebraic) multiplier ideal sheaf of a nonzero ideal sheaf $\mfa \subset \cO_X$ associated to $c > 0$ for the pair $(X,D)$ is defined by
$$
    \cJ[(X,D);\mfa^c] = \mu_* \cO_Y \Big(K_Y - \big\lfloor \frac{1}{q}\mu^*(q(K_X+D)) + cF \big\rfloor \Big) \subset \cM_X,
$$
where $q$ is a positive integer such that $q(K_X + D)$ is Cartier and $\mu:Y\to X$ is a simultaneous log resolution of $D$ and $\mfa$ with $\mfa\cdot\cO_Y=\cO_Y(-F)$ for some effective divisor $F$ on $Y$.

\begin{theorem}[cf. {\cite[Proposition~4.3.6]{an2021Thesis}}, {\cite[Remark~2.5]{li2022question}}] \label{thm:analytic_mis_extend_algebraic_mis}
    Let $\mfa$ be a nonzero ideal sheaf on a normal algebraic variety $X$ and let $D$ be a Weil $\bbQ$-divisor on $X$ such that $K_X + D$ is $\bbQ$-Cartier. If $\varphi_\mfa$ is a quasi-psh function on $X$ associated to $\mfa$ and $c$ is a positive real number, then
    $$
    \cI[(X,D);c\varphi_{\mfa}] = \cJ[(X,D);\mfa^c]^{\mathrm{an}}.
    $$
\end{theorem}
\begin{proof}
    Let $\mu: Y \to X$ be a log resolution of $D$ and $\mfa$ such that
    $$
        \mfa \cdot \cO_Y = \cO_Y(-F)
    $$
    for an effective divisor $F$ on $Y$. The assertion is local on $X$. Fix $x\in X$ and shrink to an open neighborhood $U$ on which $\omega_X^{[q]}(qD)$ has a generator $\beta$. Set
    $$
        G := \frac{1}{q} \mu^* \bigl(q(K_X + D)\bigr), \quad A := K_Y-G.
    $$
    As a meromorphic $q$-canonical form, $\mu^* \beta$ satisfies
    $$
        \frac{1}{q} \divisor_Y (\mu^* \beta) = A
    $$
    on $\mu^{-1}(U)$. For $y \in \mu^{-1}(U)$, choose local coordinates $z_1,\ldots,z_n$ centered at $y$ such that the prime divisors appearing in $\Supp(F)\cup\Supp(A)$ and passing through $y$ are $E_i = (z_i = 0)$. Write
    $$
        A=\sum_i a_i E_i,\quad F = \sum_i b_i E_i.
    $$
    Since $\mfa\cdot\cO_Y=\cO_Y(-F)$, the pull-backs of the local generators of $\mfa$ have the common monomial factor $\prod_i z_i^{b_i}$ and no further common zero. Hence
    $$
        \varphi_\mfa \circ \mu = \sum_i b_i \log |z_i| + O(1).
    $$
    The divisor identity for $\mu^*\beta$ gives
    $$
        \mu^* dV_\beta \approx \prod_i|z_i|^{2 a_i}\,d\lambda,
    $$
    where $d\lambda$ is a smooth positive volume form and $\approx$ means both sides are comparable to each other by a positive bounded function with strictly positive lower bound.
    
    Let $f \in \cM_{X,x}$ be the germ of a meromorphic function at $x$ and define the meromorphic function
    $$
        h := (f\circ\mu) \prod_i z_i^{\lceil a_i-cb_i \rceil}.
    $$
    After the change of variables, the integrand with respect to $d\lambda$ is comparable to
    $$
        |h|^2 \prod_i |z_i|^{2(a_i-cb_i-\lceil a_i-cb_i \rceil)}.
    $$
    For every $i$, we have $-1 < a_i-cb_i-\lceil a_i-cb_i \rceil \le 0$. Hence the weight function
    $$
        w := \prod_i |z_i|^{2(a_i-cb_i-\lceil a_i-cb_i \rceil)}
    $$
    is locally integrable and bounded below by a positive constant on a sufficiently small polydisk.
    If $h$ is locally square integrable with weight $w$, then $h \in L^2_{\loc}$ and hence $h$ is holomorphic.

    Conversely, if $h$ is holomorphic, it is bounded on a smaller polydisk and is therefore locally square integrable with weight $w$.

    Since $\prod_i z_i^{-\lceil a_i-cb_i \rceil}$ is a local generator of $\cO_Y(\lceil A-cF \rceil)$, the local square integrability of $h$ with weight $w$ near $y$ is therefore equivalent to
    $$
        (f \circ \mu)_y \in \cO_Y \big( \lceil A - cF \rceil \big)_y
        = \cO_Y \Big(K_Y - \big\lfloor \frac{1}{q} \mu^*(q(K_X+D)) + cF \big\rfloor \Big)_y.
    $$
    The properness of $\mu$ gives
    $$
        \cI[(X,D);c\varphi_{\mfa}] = \cJ[(X,D);\mfa^c]^{\mathrm{an}}.
    $$
\end{proof}

In the algebraic setting, the multiplier ideal sheaf $\cJ[(X,D);B]$ is an ideal sheaf of $\cO_X$ if $D$ and $B$ are effective divisors. We extend this fact to the analytic multiplier ideal sheaf $\cI[(X,D);\varphi]$ for effective $D$ and any $\varphi \in L_{\loc}^1(X^\reg)$ that is locally bounded above, in particular for quasi-psh $\varphi$.

\begin{proposition} \label{prop:effective_imply_ideal}
    Suppose that $\varphi \in L_{\loc}^1(X^\reg)$ is locally bounded above on $X$ and that $D$ is effective. Then, $\cI[(X,D);\varphi]$ is an ideal sheaf of $\cO_X$, that is,
    $$
        \cI[(X,D);\varphi] \subset \cO_X.
    $$
\end{proposition}
\begin{proof}
    Since the statement is local, suppose that, for a small open set $U$, a section $f$ in $\Gamma\big(U, \cI[(X,D);\varphi]\big)$ is not holomorphic and denote by $P$ the polar set of $f$ in $U$. The normality of $X$ guarantees that $P$ is a nonempty analytic subset of pure codimension $1$. Choose a point $x \in P \cap U$ that is a smooth point of both $X$ and $P$, and assume that $x \in (\Supp D)^\reg$ if $x \in \Supp D$. Choose local coordinates $(z_1,\ldots, z_n)$ centered at $x$ on $U$ with $P \cap U = \{ z \in U \,:\, z_1 = 0\}$ after shrinking $U$ if necessary. Set $k := -\ord_{P}(f)$. Then $g := z_1^k f$ belongs to $\cO_{X}(U)$ and does not vanish generically on $P$. Moreover, by moving the point $x$ slightly, we can assume that $g(x) \neq 0$.

    On a small relatively compact open neighborhood $W$ of $x$, write
    $$
        D = d_1 D_1,
    $$
    where $d_1 \ge 0$ and $D_1$ is a prime divisor having local equation $g_1$ with $\sup_{W}|g_1| \le C$ for some constant $C$. We may assume that $W \subset X^\reg$ and thus the meromorphic $q$-canonical form
    $$
        \beta:= g_1^{-qd_1} (dz_1 \wedge \cdots \wedge dz_n)^{\otimes q}
    $$
    generates $\omega_X^{[q]}(qD)$ on $W$. From the inequality $dV_\beta \ge 2^n C^{-2d_1} d\lambda$, we have
    $$
        \int_{D(0,\epsilon)^n} |f|^2 e^{-2\varphi}dV_\beta \ge 2^n C^{-2d_1} \int_{D(0,\epsilon)^n} |z_1^{-k}g|^2 e^{-2\varphi} d\lambda \ge C' \int_{D(0,\epsilon)} |z_1|^{-2k} d\lambda(z_1) = +\infty
    $$
    for some $C' > 0$, where $D(0,\epsilon) \subset \bbC$ is the open disk of radius $\epsilon$ centered at $0$ and $\epsilon > 0$ is taken to be sufficiently small so that the polydisk $D(0,\epsilon)^n$ is contained in $W$. Here the second inequality follows from the fact that $\varphi$ is locally bounded above and $g(x) \neq 0$, and the last equality follows from $k \ge 1$. This contradicts the assumption $f \in \cI[(X,D);\varphi]_{x}$.
\end{proof}

Just as the multiplier ideal sheaf of a quasi-psh function on a complex manifold is coherent, so is the multiplier ideal sheaf $\cI[(X,D);\varphi]$ for a quasi-psh function $\varphi$. For this, we start with the following lemma.

\begin{lemma} \label{lem:local_formula_for_snc}
   Let $\mu : Y \to X$ be a proper bimeromorphic morphism from a complex manifold $Y$ such that $\mu$ is biholomorphic over $X^\reg \backslash \Supp D$ and
    $$
        E := \Exc(\mu) \cup \mu_*^{-1} (\Supp D)
    $$
    is a divisor with simple normal crossings. Let $q$ be a positive integer such that $q(K_X+D)$ is Cartier and set
    $$
        A := K_Y - \frac{1}{q}\mu^* \big( q(K_X+D) \big),
    $$
    which is a $\bbQ$-divisor supported on $E$.  Let $\beta$ be a generator of $\omega_X^{[q]}(qD)$ on an open subset $U$ of $X$. For $y \in \mu^{-1}(U)$, choose local coordinates $z_1,\ldots,z_n$ centered at $y$ such that the components of $E$ passing through $y$ are $E_i = (z_i = 0)$, and write $A = \sum_i a_i E_i$ near $y$. Then
        $$
            \mu^* dV_\beta = e^{2\eta} \prod |z_i|^{2a_i} \, d\lambda
        $$
        near $y$, where $d\lambda$ is the Lebesgue measure of the coordinates and $\eta$ is a bounded function.
\end{lemma}
\begin{proof}
    As a meromorphic $q$-canonical form, the generator $\beta$ has divisor $-qD$ on $U \cap X^\reg$. By the local parametrization theorem, the pull-back $\mu^* \beta$ can be considered as a meromorphic $q$-canonical form on $\mu^{-1}(U)$ and it is a generator of $\mu^*\big(  \omega_X^{[q]}(qD) \big) = \omega_Y^{\otimes q} \otimes \cO_Y(-qA)$. Since $q(K_X+D)$ is Cartier, the divisor $qA$ is integral and supported on $E$. Thus, in the chosen coordinates,
    $$
        \mu^*\beta = w \cdot \Big( \prod_i z_i^{q a_i} \Big) \, ( dz_1 \wedge \cdots \wedge dz_n )^{\otimes q}
    $$
    for a nowhere vanishing holomorphic function $w$ near $y$. This yields
    $$
        \mu^*dV_\beta = \left( \sqrt{-1}^{qn^2} \mu^*\beta \wedge \overline{\mu^*\beta} \right)^{\frac{1}{q}} = e^{2\eta} \prod |z_i|^{2a_i} \, d\lambda
    $$
    with $e^{2\eta} = 2^n |w|^{2/q}$.
\end{proof}

With the above lemma, we can prove transformation formulas for multiplier ideal sheaves under modifications. These formulas are analytic counterparts of the transformation rule for algebraic multiplier ideal sheaves \cite[Proposition~9.3.62]{lazarsfeld2004positivity2}.

\begin{proposition}[cf. {\cite[Proposition~2.4]{li2022question}}] \label{prop:transformation_formula}
    Let $\mu:Y \to X$, $E$, $A$, and $q$ be as in Lemma~\ref{lem:local_formula_for_snc} and let $\varphi$ be a quasi-psh function on $X$.
    Then the following hold.
    \begin{enumerate}
        \item Let $\tilde\varphi := \varphi \circ \mu + \varphi_{\lceil A \rceil - A}$. Then $\tilde\varphi$ is quasi-psh on $Y$ and
        $$
            \cI [(X,D);\varphi] = \mu_*\Big( \cO_Y \big( \lceil A \rceil \big) \otimes \cI[(Y,0);\tilde\varphi] \Big).
        $$
        \item Let $D_Y := -A$, so that $(Y,D_Y)$ is a pair with $q(K_Y + D_Y) = \mu^*\big(q(K_X+D)\big)$. Then
        $$
            \cI[(X,D);\varphi] = \mu_*\cI[(Y,D_Y);\varphi\circ\mu].
        $$
    \end{enumerate}
\end{proposition}
\begin{proof}
    (1) Let $U$ and $y$ be as in Lemma~\ref{lem:local_formula_for_snc}. Write $\epsilon_i := \lceil a_i \rceil - a_i \in [0,1)$, so that $\lceil A \rceil - A = \sum_i \epsilon_i E_i$ is effective. Thus, near $y$, we have
    $$
        \tilde \varphi = \varphi\circ\mu + \sum \epsilon_i \log|z_i| + \mathrm{O}(1),
    $$
    which is quasi-psh.

    Let $\beta$ be the generator of $\omega_X^{[q]}(qD)$ on $U$ as in Lemma~\ref{lem:local_formula_for_snc}. Since $\mu$ is biholomorphic over $X^\reg \backslash \Supp D$ and since $E$ and $X^{\sing} \cup \Supp D$ have Lebesgue measure zero, a change of variables gives
    $$
        \int_{U \cap X^\reg} |f|^2 e^{-2\varphi} dV_\beta = \int_{\mu^{-1}(U)} |f\circ\mu|^2 e^{-2\varphi\circ\mu} \mu^* dV_\beta
    $$
    for every $f \in \cM_X(U)$. By Lemma~\ref{lem:local_formula_for_snc}, up to a bounded positive factor, the integrand on the right-hand side is
    $$
        |f\circ \mu|^2 \prod |z_i|^{2a_i} e^{-2\varphi\circ\mu} = \Big| (f\circ\mu) \prod z_i^{\lceil a_i \rceil} \Big|^2 e^{-2\tilde\varphi}.
    $$
    By Proposition~\ref{prop:effective_imply_ideal}, a meromorphic function that is square integrable with weight $e^{-2\tilde\varphi}$ must be holomorphic because $\tilde\varphi$ is locally bounded above. Since $\prod_i z_i^{-\lceil a_i \rceil}$ is a local generator of $\cO_Y(\lceil A \rceil)$, the finiteness of the last integral near $y$ is therefore equivalent to
    $$
        (f\circ\mu)_y \in \Big( \cO_Y \big( \lceil A \rceil \big) \otimes \cI[(Y,0); \tilde\varphi] \Big)_y.
    $$
    As $\mu$ is proper, the fiber $\mu^{-1}(x)$ is compact for each $x \in X$, so the above holds for every $y \in \mu^{-1}(x)$ if and only if the integral converges on $\mu^{-1}(U')$ for a sufficiently small neighborhood $U'$ of $x$. Hence, the assertion follows from $\mu_*\cM_Y = \cM_X$.

    (2) Since $\mu^*\beta$ satisfies $\divisor_Y(\mu^*\beta) = qA = -qD_Y$ as a meromorphic $q$-canonical form, it is a generator of $\omega_Y^{\otimes q} \otimes \cO_Y(qD_Y)$ on $\mu^{-1}(U)$, and $dV_{\mu^*\beta} = \mu^* dV_\beta$. Hence, the stalk of $\cI[(Y,D_Y); \varphi\circ\mu]$ at $y$ consists exactly of the germs $h \in \cM_{Y,y}$ such that $|h|^2 e^{-2\varphi\circ\mu} \mu^* dV_\beta$ is integrable near $y$. The argument is completed as in (1).
\end{proof}

The sheaf $\cI[(Y,0);\tilde\varphi]$ appearing in Proposition~\ref{prop:transformation_formula} is exactly the multiplier ideal sheaf of $\tilde\varphi$ on the complex manifold $Y$, in the sense of Nadel \cite{nadel1990multiplier}.

In addition to the transformation formulas, it is useful to compare multiplier ideal sheaves associated to different pairs whose divisors differ by an integral Cartier divisor $H$. The following lemma shows that removing $H$ from the divisor of a pair twists the multiplier ideal sheaf by a local equation of $H$.

\begin{lemma} \label{lem:twist_formula}
    Let $\varphi$ be a measurable function on $X$ that is locally integrable on $X^\reg$ and let $H$ be an integral Cartier divisor on an open subset $U \subset X$ with local equation $h \in \cM_X^*(U)$. Then,
    $$
        \cI [(X,D);\varphi]|_U = h^{-1} \cdot \cI[(U,D|_U+H);\varphi]
    $$
    as $\cO_U$-submodules of $\cM_U$.
\end{lemma}
\begin{proof}
    Since $H$ is an integral Cartier divisor on $U$, the divisor $q(K_U + D|_U + H)$ is Cartier if $q(K_X + D)$ is Cartier. Let $\beta$  be a generator of $\omega_X^{[q]}(qD)$ on $U$ after shrinking $U$ if necessary. Then, it is clear that $h^{-q}\beta$ generates $\omega_U^{[q]}\big(q(D|_U+H)\big)$. Thus, the definition of adapted volume forms gives
    $$
        dV_{\beta} = |h|^2 dV_{h^{-q}\beta}
    $$
    on $U$, which implies $f \in \cI[(X,D);\varphi]_x$ if and only if $fh \in \cI[(U,D|_U+H);\varphi]_x$ for $x \in U$.
\end{proof}

We next show that the denominators of $\cI[(X,D);\varphi]$ can be cleared locally when $\varphi$ is locally bounded above. We then prove coherence for quasi-psh $\varphi$.

\begin{proposition} \label{prop:fractional_ideal_sheaf}
    Let $\varphi \in L_{\loc}^1(X^\reg)$ be a measurable function on $X$ that is locally bounded above.
    Then for every $x \in X$, there exist an open neighborhood $U$ of $x$ and a nonzero holomorphic function $h \in \cO_X(U)$ such that
    $$
        h \cdot \cI[(X,D);\varphi]|_U \subset \cO_U.
    $$
\end{proposition}
\begin{proof}
    For $x \in X$, one can find a Stein open neighborhood $U$ of $x$ that intersects only finitely many prime divisors appearing in $D$. Write $D = \sum_i d_i D_i$ on $U$, where $D_i$ are prime divisors with $D_i \cap U \neq \emptyset$. By Cartan's theorem A, there exists a nonzero holomorphic function $h \in \cO_X(U)$ such that $D + \divisor (h) \ge 0$ on $U$. Set $H := \divisor (h)$. By Lemma~\ref{lem:twist_formula}, we have
    $$
        h \cdot \cI[(X,D);\varphi]|_U = \cI[(U,D|_U+H);\varphi].
    $$
    Thanks to Proposition~\ref{prop:effective_imply_ideal}, the sheaf $\cI[(U,D|_U+H);\varphi]$ is an ideal sheaf of $\cO_U$ since $\varphi$ is locally bounded above and $D|_U + H$ is effective. Therefore, $h \cdot \cI[(X,D);\varphi]|_U \subset \cO_U$, as desired.
\end{proof}

\begin{corollary} \label{cor:multiplier_ideal_coherence}
    For a quasi-psh function $\varphi$ on $X$, the multiplier ideal sheaf $\cI[(X,D);\varphi]$ is a coherent analytic $\cO_X$-submodule of $\cM_X$.
\end{corollary}
\begin{proof}
    Since coherence is a local property, we may assume that $X$ is small enough that $\omega_X^{[q]}(qD)$ admits a global generator $\beta$. Let $\mu: Y \to X$, $E$, $A$ be as in Lemma~\ref{lem:local_formula_for_snc}; such a $\mu$ exists after shrinking $X$, by resolution of singularities for complex analytic spaces \cite{bierstone1997canonical}. We follow the notation of Proposition~\ref{prop:transformation_formula}.

    Since $\tilde\varphi$ is a quasi-psh function on $Y$, the multiplier ideal sheaf $\cI[(Y,0);\tilde\varphi]$ is just the multiplier ideal sheaf of $\tilde\varphi$ on the complex manifold $Y$, which is coherent \cite{nadel1990multiplier}. Being an integral Cartier divisor, the divisor $\lceil A \rceil$ gives an invertible sheaf $\cO_Y\big(\lceil A \rceil\big)$. By Proposition~\ref{prop:transformation_formula} and Grauert's direct image theorem \cite{grauert1960theorem},
    $$
        \cI[(X,D);\varphi] = \mu_* \left( \cO_Y\big(\lceil A \rceil\big) \otimes \cI[(Y,0);\tilde\varphi] \right)
    $$
    is coherent.
\end{proof}

The transformation formula reduces the strong openness property of the multiplier ideal sheaf $\cI[(X,D);\varphi]$ to the strong openness theorem of Guan--Zhou on the complex manifold $Y$ \cite{guan2015proof}. We next establish the resulting strong openness property.

\begin{corollary}[Strong openness] \label{cor:strong_openness}
    Let $\varphi$ be a quasi-psh function on $X$. Then
    $$
        \cI[(X,D);\varphi] = \bigcup_{\epsilon > 0} \cI[(X,D);(1+\epsilon)\varphi].
    $$
\end{corollary}
\begin{proof}
    Fix $x \in X$. After shrinking $X$ around $x$ and subtracting a constant from $\varphi$, we may assume that $\varphi \le 0$. Since $(1+\epsilon)\varphi \le \varphi$, one inclusion is immediate. For the reverse inclusion, let $f \in \cI[(X,D);\varphi]_x$. We keep the notation of Proposition~\ref{prop:transformation_formula}, and for $\epsilon \ge 0$, write
    $$
        \tilde\varphi_\epsilon := (1+\epsilon) \varphi\circ\mu + \varphi_{\lceil A \rceil - A},
    $$
    so that Proposition~\ref{prop:transformation_formula}, applied to the quasi-psh function $(1+\epsilon)\varphi$, gives
    $$
        \cI[(X,D);(1+\epsilon)\varphi] = \mu_* \Big( \cO_Y \big( \lceil A \rceil \big) \otimes \cI[(Y,0);\tilde\varphi_\epsilon] \Big).
    $$
    Let $y \in \mu^{-1}(x)$. After subtracting sufficiently large constants, we may assume that $\varphi\circ\mu \le 0$ and $\varphi_{\lceil A \rceil - A} \le 0$ near $y$. This does not affect the corresponding multiplier ideal sheaves. Thus
    $$
        (1+\epsilon) \tilde\varphi_0 = (1+\epsilon) \varphi\circ\mu + (1+\epsilon) \varphi_{\lceil A \rceil - A} \le \tilde\varphi_\epsilon \le \tilde\varphi_0.
    $$
    Hence
    $$
        \cI[(Y,0);(1+\epsilon)\tilde\varphi_0] \subset \cI[(Y,0);\tilde\varphi_\epsilon] \subset \cI[(Y,0);\tilde\varphi_0].
    $$
    Since $\cI[(Y,0);\tilde\varphi_0] = \bigcup_{\epsilon>0} \cI[(Y,0);(1+\epsilon)\tilde\varphi_0]$ by the strong openness theorem \cite{guan2015proof}, we conclude that
    $$
        \cI[(Y,0);\tilde\varphi_0] = \bigcup_{\epsilon > 0} \cI[(Y,0);\tilde\varphi_\epsilon].
    $$
    Consequently, for each $y \in \mu^{-1}(x)$, there exists $\epsilon_y > 0$ such that $(f\circ\mu)_y$ belongs to the stalk of $\cO_Y\big(\lceil A \rceil\big) \otimes \cI[(Y,0);\tilde\varphi_{\epsilon_y}]$ at $y$. Since $\mu^{-1}(x)$ is compact and $\cI[(Y,0);\tilde\varphi_\epsilon]$ increases as $\epsilon$ decreases, we can find finitely many points $y_1,\ldots, y_k$ and open neighborhoods $V_i$ of $y_i$ covering $\mu^{-1}(x)$ such that $f\circ\mu$ is a section of $\cO_Y\big(\lceil A \rceil\big) \otimes \cI[(Y,0);\tilde\varphi_{\epsilon_{y_i}}]$ on $V_i$. Set $\epsilon := \min_i \epsilon_{y_i}$. Since $\mu$ is proper, there exists an open neighborhood $U$ of $x$ such that $\mu^{-1}(U) \subset \bigcup V_i$. Thus, we have $f \in \cI[(X,D);(1+\epsilon)\varphi]_x$.
\end{proof}

\begin{remark}
    For a quasi-psh function $\varphi$ on a complex analytic space $X$ that may be neither normal nor $\bbQ$-Gorenstein, Li has introduced two multiplier ideal sheaves attached to $X$ itself rather than to a pair: the \emph{Grauert--Riemenschneider multiplier ideal sheaf} $\cI_{\mathrm{GR}}(\varphi)$ in \cite{li2021briancon} and the \emph{Nadel-type multiplier ideal sheaf} $\cI_{\mathrm{NO}} (\varphi)$ in \cite{li2024nadel}, the latter via the Ohsawa measure. The coherence and strong openness properties are obtained there as well. These sheaves are defined without adapted measures and are different from $\cI[(X,0);\varphi]$ in general. For $X$ normal and $\bbQ$-Gorenstein one has only
    $$
        \cI_{\mathrm{NO}}(\varphi) \subset \cI[(X,0);\varphi] \subset \cI_{\mathrm{GR}}(\varphi),
    $$
    and both inclusions may be strict (see \cite[Remark~3.3]{li2021grauert}). These notions coincide with $\cI[(X,0);\varphi]$ when $X$ is a normal local complete intersection.
\end{remark}

We also recall the maximal non-lc ideal sheaf.

\begin{definition}[cf. {\cite[Definition~7.1]{fujino2011supplements}}] \label{def:maximal_non-lc_ideal}
    Let $X$ be a normal complex analytic space and let $D$ be an effective Weil $\bbQ$-divisor such that $K_X+D$ is $\bbQ$-Cartier. Working  locally on $X$, choose a positive integer $q$ such that $q(K_X+D)$ is Cartier. Let $\mu:Y\to X$ be a log resolution of $(X,D)$ and set
    $$
        A := K_Y-\frac{1}{q}\mu^* \bigl(q(K_X+D)\bigr) = \sum_i a_i E_i, \quad F:=\sum_{a_i\le -1}E_i.
    $$
    The {\bf maximal non-lc ideal sheaf} of $(X,D)$ is defined by
    $$
        \cJ'(X,D) = \mu_* \cO_Y \bigl( \lceil A+\epsilon F\rceil \bigr)
    $$
    for sufficiently small $\epsilon > 0$.
\end{definition}

The locally defined ideal sheaves $\cJ'(X,D)$ glue to a well-defined ideal sheaf on $X$. The discrepancy divisor $A$ is a $\bbQ$-divisor, with $qA$ integral. For $0<\epsilon<1/q$, we have
$$
    \lceil A+\epsilon F \rceil = \lceil A \rceil + \sum_{\substack{a_i \in \bbZ \\ a_i \le -1}}E_i.
$$
Thus the perturbation changes the round-up precisely at the negative integral discrepancies.

The resolution-independence argument in \cite[Lemma~7.3]{fujino2011supplements} applies in the analytic setting as well. The resulting sheaf is a coherent ideal sheaf, independent of sufficiently small $\epsilon > 0$, and it detects log canonical singularities.

\begin{proposition}[{\cite[Lemma~7.2]{fujino2011supplements}}]
    Let $(X,D)$ be as in Definition~\ref{def:maximal_non-lc_ideal}.
    Then $(X,D)$ is lc if and only if $\cJ'(X,D) = \cO_X$.
\end{proposition}

The maximal non-lc ideal sheaf $\cJ'(X,D)$ can also be expressed in terms of analytic multiplier ideal sheaves attached to the pair $(X,D)$.

\begin{proposition}[{\cite[Lemma~7.5]{fujino2011supplements}}] \label{prop:analytic_characterization_of_maximal_non-lc_ideal}
    Let $(X,D)$ be as in Definition~\ref{def:maximal_non-lc_ideal}, and set $\mathfrak{b} := \cI[(X,D);0]$.
    Then, locally on $X$,
    $$
        \cJ'(X,D) = \cI[(X,D);-\epsilon \varphi_{\mathfrak{b}}]
    $$
    for sufficiently small $\epsilon > 0$.
\end{proposition}

These statements follow from the log resolution computations in \cite[Lemmas~7.2 and~7.5]{fujino2011supplements}, together with the adapted volume formula in Lemma~\ref{lem:local_formula_for_snc}.

\section{Legendre transforms and recession functions} \label{sec:legendre_transforms_and_recession_functions}

In this section, we recall some basic notions in convex analysis, the Legendre transform and the recession function, together with their basic properties. We then establish a criterion for the integrability of an exponential weight over a polyhedral cone. This criterion will play a crucial role in the proof of our main theorem.

Let $g$ be a function on a subset $C \subset \bbR^n$ that takes its values in $\bbR \cup \{\infty\}$. We call it a {\bf convex function} if the epigraph of $g$ is a convex subset of $\bbR^{n+1}$. Equivalently, $g$ is convex if its domain $\dom g := \{x \in C\,:\, g(x) < \infty \}$ is a convex subset of $\bbR^n$ and
$$
g(\lambda x + (1-\lambda)y) \le \lambda g(x) + (1-\lambda)g(y)
$$
whenever $x,y \in \dom g$ and $\lambda \in [0,1]$.

Extending $g$ to $\bbR^n$ by assigning $\infty$ to $\bbR^n \backslash \dom g$, we regard $g$ as a convex function on $\bbR^n$. Assume that $\dom g \neq \emptyset$, namely, $g$ is proper. The {\bf Legendre transform} of $g$ is the function $g^*$ defined by
$$
g^*(v) = \sup_{x \in \bbR^n} \big( \inp{v}{x} - g(x) \big).
$$

\begin{proposition}
    Let $g$ be a proper convex function on $\bbR^n$.
    \begin{enumerate}
        \item The Legendre transform of $g$ is a proper lower semicontinuous convex function on $\bbR^n$.
        \item The biconjugate $g^{**}$ of $g$ is the greatest proper lower semicontinuous convex function such that $g^{**} \le g$. In particular, $g^{**} = g$ if and only if $g$ is a lower semicontinuous convex function.
    \end{enumerate}
\end{proposition}

In what follows, we will use $\bbR^r$ and $N_\bbR$ interchangeably. For a proper lower semicontinuous convex function $g: N_\bbR \to \bbR \cup \{\infty\}$, we define the {\bf recession function} of $g$ by
$$
    (\rec g) (u) = \lim_{\lambda \to \infty} \frac{g(x_0 + \lambda u) - g(x_0)}{\lambda},
$$
where $x_0$ is any point in the domain of $g$. The recession function $\rec g$ of $g$ is independent of the choice of the reference point in $\dom g$. 

For a nonempty set $C$ in $N_\bbR$, the {\bf indicator function} of $C$ is defined by
$$
\delta_C (x) =
\left\{
\begin{array}{*2{>{\displaystyle}l}}
    0       &   \text{if $x \in C$,} \\
    \infty  &   \text{if $x \not\in C$,}
\end{array}
\right.
$$
while the {\bf support function} of $C$ is given by
$$
    \delta_C^* (v) = \sup_{x \in C} \inp{v}{x}.
$$
As the notation suggests, $\delta^*_C$ is the Legendre transform of the indicator function $\delta_C$. For the converse, one can easily check that
$$
    (\delta_{C}^*)^* = \delta_{\overline{\conv C}}.
$$

The Legendre transform and the recession function describe asymptotic behaviors of a convex function from different perspectives and are related in the following way.

\begin{theorem}[{\cite[Theorem~13.3]{rockafellar1997convex}}] \label{thm:support_function_and_recession_function}
    Let $g$ be a proper lower semicontinuous convex function on $N_\bbR$. Then, the support function of $\dom g^*$ is the recession function $\rec g$ of $g$.
\end{theorem}

Let $g$ be a finite convex function on $N_\bbR$ and let $\sigma$ be a strongly convex rational polyhedral cone in $N_\bbR$. Our main criterion is expressed in terms of the domain of the Legendre transform of $g + \delta_{-\sigma}$. We therefore record several properties of this set for later use.

\begin{theorem} \label{thm:inclusions_between_g_plus_indicator_and_domain_plus_dual_cone}
    Let $\sigma$ be a closed convex cone in $N_\bbR$ and let $g$ be a finite convex function on $N_\bbR$. Then,
    $$
        \dom g^* + \sigma^\vee \subset \dom (g + \delta_{-\sigma})^* \subset \overline{\dom g^* + \sigma^\vee}.
    $$
\end{theorem}
\begin{proof}
    For $v \in \dom g^*$ and $m \in \sigma^\vee$, observe that $\inp{m}{x}-\delta_{-\sigma}(x) \le 0$. Thus
    $$
        \inp{v+m}{x} - (g(x) + \delta_{-\sigma}(x)) \le \inp{v}{x} - g(x) \le g^*(v) < \infty.
    $$
    For the second inclusion, recall that
    $$
        (g_1 + \cdots + g_p)^* = \operatorname{cl}(g_1^* \,\square\, \cdots \,\square\, g_p^*),
    $$
    where $g_1,\ldots,g_p$ are proper lower semicontinuous convex functions on $N_\bbR$ whose domains have a point in common. Here $\square$ denotes infimal convolution and $\operatorname{cl}$ denotes lower semicontinuous closure (see \cite[Theorem~16.4]{rockafellar1997convex}).
    In our case, we have
    $$
        (g + \delta_{-\sigma})^* = \operatorname{cl}(g^* \,\square\, \delta_{-\sigma}^*) = \operatorname{cl}(g^* \,\square\, \delta_{\sigma^\vee}).
    $$
    The second equality follows from $\delta_{-\sigma}^* = \delta_{\sigma^\vee}$. Since
    $$
        (g^* \,\square\, \delta_{\sigma^\vee}) (v) = \inf_{m \in M_\bbR}(g^*(m) + \delta_{\sigma^\vee}(v-m)),
    $$
    the domain of $g^* \,\square\, \delta_{\sigma^\vee}$ is contained in $\dom g^* + \sigma^\vee$, which implies 
    $$
        \dom (g + \delta_{-\sigma})^* = \dom (\mathrm{cl} (g^* \,\square\, \delta_{\sigma^\vee})) \subset \overline{\dom g^* + \sigma^\vee}.
    $$
\end{proof}

\begin{lemma} \label{lem:strict_recession_inequality_and_domain_of_Legendre_transform}   
    Let $\sigma$ be a closed convex cone in $N_\bbR$ and let $g$ be a finite convex function on $N_\bbR$. Then, the recession function of $g$ determines the closure and the interior of the domain of $(g + \delta_{-\sigma})^*$, that is,
    \begin{gather*}
        \overline{\dom (g+\delta_{-\sigma})^*} = \big\{ m \in M_\bbR \,:\, \text{$\inp{m}{x} \le (\rec g)(x)$ for all $x \in -\sigma$} \big\},\\
        \interior \big( \dom (g+\delta_{-\sigma})^* \big) = \big\{ m \in M_\bbR \,:\, \text{$\inp{m}{x} < (\rec g)(x)$ for all nonzero $x \in -\sigma$} \big\}.
    \end{gather*}
\end{lemma}
\begin{proof}
    Write $\tilde g := g + \delta_{-\sigma}$. By Theorem~\ref{thm:support_function_and_recession_function}, the recession function $\rec \tilde g$ of $\tilde g$ coincides with $\delta_{\dom \tilde g^*}^*$, and therefore
    $$
        (\rec \tilde g)^* = \delta_{\overline{\dom \tilde g^*}}.
    $$
    This implies that $m \in \overline{\dom \tilde g^*}$ if and only if
    \begin{equation} \label{ineq:conjugate_of_recession_g}
        \sup_{x \in N_\bbR} \big( \inp{m}{x} - (\rec \tilde g)(x) \big) = 0.
    \end{equation}
    A direct computation shows
    $$
        (\rec \tilde g)(x) =
        \left\{
            \begin{array}{*2{>{\displaystyle}l}}
                (\rec g)(x)     &   \text{if $x \in -\sigma$,} \\
                \infty          &   \text{if $x \not\in -\sigma$.}
            \end{array}
        \right.
    $$
    Thus, $m \in \overline{\dom \tilde g^*}$ if and only if
    $$
        \sup_{x \in -\sigma} \big( \inp{m}{x} - (\rec g)(x) \big) = 0.
    $$
    The equality involving the interior follows immediately.
\end{proof}

\begin{remark}
    The set $\big\{ m \in M_\bbR \,:\, \text{$\inp{m}{x} \le (\rec g)(x)$ for all $x \in -\sigma$} \big\}$ generalizes the \emph{indicator diagram} of a toric psh function in \cite{rashkovskii2012extreme,rashkovskii2013multi}. In that setting, the first equality in the lemma shows that the indicator diagram coincides with the closure of the Newton convex body considered in \cite{guenancia2012toric}. Furthermore, the first equality extends the characterization in \cite[Lemma~1.19]{guenancia2012toric} to closed convex cones.
\end{remark}

The strict inequality involving the recession function in Lemma~\ref{lem:strict_recession_inequality_and_domain_of_Legendre_transform} also characterizes the integrability of an associated exponential weight over $-\sigma$. This criterion generalizes a result of Guenancia. 

\begin{theorem}[cf. {\cite[Proposition~1.9]{guenancia2012toric}}] \label{thm:negative_recession_iff_integrability}
    Let $\sigma$ be a closed convex cone of full dimension in $N_\bbR$, and let $g$ be a finite convex function on $N_\bbR$.
    For $m \in M_\bbR$, define a function $h_m$ on $N_\bbR$ by
    $$
        h_m(x) = \inp{m}{x} - g(x).
    $$
    Then, $\inp{m}{x} < (\rec g)(x)$ for all nonzero $x \in -\sigma$ if and only if
    $$
        \int_{-\sigma} e^{2h_m(x)} \, dx < \infty.
    $$
\end{theorem}
\begin{proof}    
    Assume that $\inp{m}{x} < (\rec g)(x)$ for all nonzero $x \in -\sigma$. Define an auxiliary function $F_m$ on $N_\bbR$ by
    $$
        F_m(x) = g(x) - \inp{m}{x} + \delta_{-\sigma}(x).
    $$
    The function $F_m$ is a proper lower semicontinuous convex function on $N_\bbR$, and 
    $$
        (\rec F_m)(x) =
        \left\{
            \begin{array}{ll}
                (\rec g)(x) - \inp{m}{x}, & \text{if $x \in -\sigma$,} \\
                +\infty,    & \text{otherwise.}
            \end{array}        
        \right.
    $$
    Therefore, $(\rec F_m)(x) > 0$ for all $x \neq 0$. By \cite[Theorem~3.26]{rockafellar1998variational}, $F_m$ is level-coercive, so there exist $\gamma > 0$ and $C \in \bbR$ such that
    $$
            F_m(x) \ge \gamma \Vert x \Vert - C
    $$
    for all $x \in N_\bbR$. Since $h_m \equiv -F_m$ on $-\sigma$, we have
    \begin{equation} \label{ineq:upper_bound_from_recession}
        h_m(x) \le C - \gamma \Vert x \Vert
    \end{equation}
    for all $x \in -\sigma$. The inequality (\ref{ineq:upper_bound_from_recession}) guarantees that
    $$
        \int_{-\sigma} e^{2h_m(x)} \, dx < \infty.
    $$
    For the converse, suppose that $\inp{m}{x_0} \ge (\rec g)(x_0)$ for some nonzero $x_0 \in -\sigma$. Since the value of the recession function of $g$ is an increasing limit, the function $g - \inp{m}{\cdot}$ has an upper bound along the ray $\bbR_{\ge 0} x_0$. In other words, the function $h_m$ has a lower bound along $\bbR_{\ge 0} x_0$. Let $K$ be a compact subset of $-\sigma$ with nonempty interior. Consider a half-tube $T \subset -\sigma$ defined by
    $$
        T = \Bigl\{ \lambda x_0 + \frac{1}{2}x \,:\, \lambda \ge 1, \, x \in K\Bigr\}.
    $$
    By the concavity of $h_m$,
    $$
        h_m\Bigl(\lambda x_0 + \frac{1}{2}x\Bigr) \ge \frac{1}{2}h_m(2\lambda x_0) + \frac{1}{2}h_m(x).
    $$
    The right-hand side is bounded below uniformly in $\lambda$ and $x \in K$, and thus $e^{2h_m}$ is bounded below by a positive constant on $T$, which has infinite volume. Hence the integral diverges.
\end{proof}

\begin{remark} \label{rem:uniform_coercivity}
    Let $\sigma$ be as in Theorem~\ref{thm:negative_recession_iff_integrability}, and let $g$ be a finite convex function on $N_\bbR \times \bbR^k$.
    Fix $m \in M_\bbR$ and set $h_m (x,y) := \inp{m}{x} - g(x,y)$. As in the theorem, assume that 
    $$
        \inp{m}{x} < (\rec g)(x,0)
    $$
    for all nonzero $x \in -\sigma$. Then, for every compact subset $K \subset \bbR^k$, one can find uniform constants $\gamma > 0$ and $\beta$ such that
    $$
        h_m(x,y) \le \beta - \gamma \Vert x \Vert
    $$
    for all $x \in -\sigma$ and $y \in K$. This follows by applying the coercivity argument in the proof of Theorem~\ref{thm:negative_recession_iff_integrability} to $g(x,y)-\inp{m}{x}+\delta_{-\sigma}(x)+\delta_{\conv K}(y)$. We will use this uniform estimate in the proof of our main results.
\end{remark}

\section{Newton convex bodies and invariant analytic sheaves} \label{sec:newton_convex_bodies_and_invariant_analytic_sheaves}

Let $\sigma$ be a strongly convex rational polyhedral cone in an $r$-dimensional vector space $N_\bbR$. Fix a basis $\{e_1,\ldots,e_r\}$ of $N$ with dual basis $e_1^{*},\ldots,e_r^{*} \in M$. Write $t_j=\chi^{e_j^*}$. Assigning $(t_1,\ldots, t_r)$ to $t \in T_N$ gives an isomorphism $T_N \simeq (\bbC^*)^r$. Define
\[
\Log : T_N \longrightarrow N_\bbR, \quad
\Log(t)=x:=\sum_{j=1}^{r} ( \log|t_j| )\, e_j .
\]
It is worth noting that the point $x=\Log(t)$ is determined, independently of the chosen basis, by
\begin{equation}\label{eqn:log_representation}
\inp{m}{x}=\log | \chi^{m}(t) |,
\end{equation}
for all $m\in M$. Indeed, for $m=\sum_j {m_j e_j^{*}}$, both sides equal $\sum_j {m_j\log|t_j|}$.

\begin{lemma} \label{lem:toric_psh_and_convex_functions}
    Let $\varphi$ be a toric psh function on $X_\sigma$. There exists a unique finite convex function $g: N_\bbR \to \bbR$ such that $g(\Log(z)) = \varphi(z)$ for all $z \in T_N$ and $g$ is increasing in $\sigma$-direction. Conversely, if $h$ is a finite convex function on $N_\bbR$ and increases in $\sigma$-direction, then the function $h \circ \Log$ uniquely extends to a toric psh function on $X_\sigma$.
\end{lemma}

\begin{proof}
    Let $K_N$ be the maximal compact torus in $T_N$. To prove that $g$ is well-defined, it suffices to show that, for any $z,w \in T_N$, $\Log (z) = \Log (w)$ if and only if there exists $k \in K_N$ such that $w = k \cdot z$. Indeed, by (\ref{eqn:log_representation}), $\Log (z) = \Log (w)$ if and only if $\log |\chi^{m}(z)| = \log |\chi^{m}(w)|$ for all $m \in M$. By the definition of the character, the latter is equivalent to saying that $wz^{-1} \in K_N$. The fact that $g$ is convex follows immediately from log-convexity of the average of $\varphi$.

    Suppose that $m_1,\ldots,m_p$ generate $\sigma^\vee \cap M$. For $R \ge 1$, define
    $$
        D_R = \{z\in X_\sigma \,:\, \text{$|\chi^{m_i}(z)| \le R$ for all $i$} \}, \quad
        S_R = \{x\in N_\bbR \,:\, \text{$\inp{m_i}{x} \le \log R$ for all $i$} \}.
    $$ 
    Being the preimage of a closed polydisk under the closed embedding of $X_\sigma$ into $\bbC^p$, the set $D_R$ is compact. Also, (\ref{eqn:log_representation}) shows that a point $z \in T_N$ belongs to $D_R$ if and only if $\Log(z) \in S_R$. Since $\varphi$ is upper semicontinuous, we have $C_R := \sup_{D_R} \varphi < \infty$, and therefore $g \le C_R$ on $S_R$. Now fix $x \in N_\bbR$ and $v \in \sigma$, and choose $R \ge 1$ so that 
    $$
        \log{R}\ge \max \{0,\inp{m_1}{x},\ldots,\inp{m_p}{x}\}.
    $$
    Since $m_i\in \sigma^\vee$, we have $\inp{m_i}{x - tv} \le \inp{m_i}{x}$ for all $t\ge 0$, so the ray $\{ x - tv \,:\, t \ge 0 \}$ is contained in $S_R$. Hence, $g(x + tv)$ is a finite convex function on $\bbR$ bounded above on $(-\infty, 0]$. Such a function is increasing. In particular, $g(x) \le g(x+v)$.

    Conversely, let $h$ be a finite convex function increasing in $\sigma$-direction and set $\varphi_0 = h \circ \Log$ on $T_N$. The function $\varphi_0$ is a $K_N$-invariant psh function. Indeed, $h$ can be written as
    $$
        h = \sup \{ \ell:N_\bbR \to \bbR \,|\, \text{$\ell$ is affine and $\ell \le h$} \}
    $$
    so that $\varphi_0$ is the upper envelope of a family of psh functions. In particular, $\varphi_0$ is psh.

    Assume that $m_1,\ldots,m_p$ generate $\sigma^\vee \cap M$. For $R \ge 1$, consider $D_R$ and $S_R$ defined above. The recession cone of $S_R$ is $-\sigma$, so the Minkowski--Weyl theorem gives $S_R = P_R - \sigma$ for some polytope $P_R$. Hence
    $$
        \textstyle\sup_{S_R} h \le \mathop{\vphantom{\sup}\max}_{P_R} h < \infty.
    $$
    Each point of $X_\sigma \backslash T_N$ lies in the interior of $D_R$ for sufficiently large $R \ge 1$. Since $\Log(D_R \cap T_N) = S_R$, the above inequality shows that $\varphi_0$ is bounded above on  $D_R \cap  T_N$. Thus $\varphi_0$ is locally bounded above near $X_\sigma \backslash T_N$. Therefore, Theorem~\ref{thm:extension_of_psh_across_analytic_subset} guarantees that $\varphi_0$ extends uniquely to a psh function $\varphi$ on $X_\sigma$. Its uniqueness and the $K_N$-invariance of $\varphi_0$ imply that $\varphi$ is toric.
\end{proof}

\begin{definition}
    Let $\varphi$ be a toric psh function on $X_\sigma$ and $g$ its associated convex function on $N_\bbR$. The {\bf Newton convex body} $P(\varphi)$ of $\varphi$ is defined by
    $$
        P(\varphi) = \{ m \in M_\bbR\,:\, \sup_{x \in -\sigma} \bigl( \inp{m}{x} - g(x) \bigr) < +\infty \}.
    $$
\end{definition}

Note that the Newton convex body $P(\varphi)$ of $\varphi$ need not be closed. However, the boundary of $P(\varphi)$ does not affect whether a (meromorphic) monomial belongs to the multiplier ideal sheaf of $\varphi$. Proposition~\ref{prop:equivalent_def_of_newton_convex_body_and_other_characterization} restates Lemma~\ref{lem:strict_recession_inequality_and_domain_of_Legendre_transform} for $P(\varphi)$.

\begin{proposition} \label{prop:equivalent_def_of_newton_convex_body_and_other_characterization}
    Let $\sigma$ be a strongly convex rational polyhedral cone in $N_\bbR$ and let $\varphi$ be a toric psh function on $X_\sigma$. Denote by $g$ its associated convex function on $N_\bbR$. Then,
    \begin{gather*}
        \overline{P(\varphi)} = \big\{ m \in M_\bbR \,:\, \text{$\inp{m}{x} \le (\rec g)(x)$ for all $x \in -\sigma$} \big\},\\
        \interior \big( P(\varphi) \big) = \big\{ m \in M_\bbR \,:\, \text{$\inp{m}{x} < (\rec g)(x)$ for all nonzero $x \in -\sigma$} \big\}.
    \end{gather*}
\end{proposition}

When $\sigma$ is full-dimensional, An defines the Newton convex body by taking an intersection over the smooth affine charts of a toric resolution \cite[Definition~4.4.8]{an2021Thesis}. The following proposition identifies this construction with the definition above.

\begin{proposition} \label{prop:comparison_with_An}
    Assume that $\sigma$ is full-dimensional. Let $\varphi$ be a toric psh function on $X_\sigma$, and let $\pi : X_\Delta \to X_\sigma$ be a toric resolution obtained by a finite smooth subdivision of $\sigma$, with maximal cones $\sigma_1,\ldots,\sigma_k$. Write
    $$
        \sigma_i = \cone (v_{i,1},\ldots,v_{i,r})
    $$
    for each $i$, where $v_{i,1},\ldots,v_{i,r}$ form a basis of $N$. Denote by $\pi_i : \bbC^r \simeq X_{\sigma_i} \to X_\sigma$ the restriction of $\pi$, and define the coordinate isomorphism $\mu_i : M_\bbR \rightarrow \bbR^r$ by
    $$
        \mu_i(m) = ( \inp{m}{v_{i,1}},\ldots,\inp{m}{v_{i,r}} ).
    $$
    Then
    $$
        P(\varphi)=\bigcap_{i=1}^k\mu_i^{-1}\bigl(P(\varphi\circ\pi_i)\bigr),
    $$
    where the Newton convex body on each smooth chart is defined as in \cite[Definition~4.4.4]{an2021Thesis}. In particular, $P(\varphi)$ agrees with An's Newton convex body and the intersection is independent of the chosen subdivision.
\end{proposition}
\begin{proof}
    Let $g$ be the convex function associated to $\varphi$. For each $i$, define the linear transformation $A_i: \bbR^r \to N_\bbR$ by
    $$
        A_i(y) := \sum_{j=1}^r y_j v_{i,j}, \quad y \in \bbR^r.
    $$
    In the logarithmic coordinates of the chart $X_{\sigma_i} \simeq \bbC^r$, the convex function associated to $\varphi\circ\pi_i$ is $g_i=g\circ A_i$. Moreover, $A_i (\bbR_{\ge0}^r) = \sigma_i$ and $\inp{\mu_i(m)}{y}  = \inp{m}{A_i(y)}$. Consequently,
    $$
        \sup_{y \in -\bbR_{\ge 0}^r} (\inp{\mu_i(m)}{y} - g_i(y)) = \sup_{x \in -\sigma_i} (\inp{m}{x}-g(x)).
    $$
    By the smooth-chart definition in \cite[Definition~4.4.4]{an2021Thesis} and the continuity of $g_i$, this value is finite if and only if $\mu_i(m) \in P(\varphi \circ \pi_i)$. Since $\sigma$ is the union of $\sigma_1,\ldots,\sigma_k$,
    $$
        \sup_{x \in -\sigma} (\inp{m}{x}-g(x)) = \max_{1 \le i \le k} \sup_{x \in -\sigma_i}(\inp{m}{x}-g(x)).
    $$
    Thus the conditions on all smooth charts are equivalent to $m \in P(\varphi)$.
\end{proof}

Let $N_\sigma$ be the intersection of $N$ with the $\bbR$-linear span of $\sigma$ in $N_\bbR$. By choosing a splitting of the short exact sequence
$$
    0 \longrightarrow N_\sigma \longrightarrow N \longrightarrow N/N_\sigma \longrightarrow 0,
$$
we obtain a lattice decomposition of $N$:
\begin{equation} \label{eqn:decomposition_of_N}
    N = N_\sigma \oplus N'.
\end{equation}
Accordingly, its dual and their scalar extensions are decomposed as follows:
\begin{equation} \label{eqn:decompositions_of_N_and_M}
\begin{gathered}
    M = M_\sigma \oplus M', \\
    N_\bbR = N_{\sigma, \bbR} \oplus N_\bbR', \\
    M_\bbR = M_{\sigma, \bbR} \oplus M_\bbR'.
\end{gathered}    
\end{equation}

These decompositions give
$$
    X_\sigma = X_{\sigma, N_\sigma} \times T_{N'}.
$$
In addition, the orbit $\orb(\sigma)$ corresponding to $\sigma$ in $N$ is the product of $T_{N'}$ and the unique fixed point $z_0$ of $X_{\sigma, N_\sigma}$, that is,
$$
    \orb(\sigma) = \{z_0\} \times T_{N'}.
$$

The algebraic monomiality of torus-invariant ideals is standard.
For the analytic setting, we record the following elementary
equivariant local-to-global principle and include the proof for
completeness. Here an analytic ideal sheaf $\mfa \subset \cO_{X_\sigma}$ is called $T_N$-invariant if
$$
    \tau_t^* \mfa = \mfa
$$
for every translation $\tau_t : X_\sigma \to X_\sigma$ by $t \in T_N$. The $T_N$-invariance of analytic $\cO_{X_\sigma}$-submodules of $\cM_{X_\sigma}$ is defined similarly.

\begin{theorem} \label{thm:local_to_global}
    Let $\mfa$ be a coherent $T_N$-invariant analytic ideal sheaf of $\cO_{X_\sigma}$ and let $z$ be a point in $\orb(\sigma)$. Suppose that the stalk $\mfa_z$ is generated by the germs  at $z$ of characters $\chi^{m_1},\ldots,\chi^{m_s}$, that are holomorphic on a neighborhood of $z$, where $m_i \in M$. Then each $\chi^{m_i}$ extends to a global holomorphic monomial on $X_\sigma$, and $\mfa$ is globally generated by $\chi^{m_1},\ldots,\chi^{m_s}$.
\end{theorem}

\begin{proof}
    For every ray $\rho$ of $\sigma$, the orbit $\orb(\sigma)$ is contained in the invariant prime divisor $D_\rho$. Since $\chi^{m_i}$ is holomorphic near $z$, it has no pole along $D_\rho$, and hence
    $$
        \inp{m_i}{v_\rho} = \ord_{D_\rho}(\chi^{m_i}) \ge 0,
    $$
    where $v_\rho$ is the primitive generator of $\rho$. Consequently, $m_i \in \sigma^\vee \cap M$, so $\chi^{m_i}$ extends to a global holomorphic monomial on $X_\sigma$.

    Let $\mathfrak{b} \subset \cO_{X_\sigma}$ be the ideal sheaf globally generated by $\chi^{m_1},\ldots,\chi^{m_s}$. Since $\mfa_z=\mathfrak{b}_z$ and both ideal sheaves are coherent, there is an open neighborhood $U$ of $z$ such that
    $$
        \mfa|_U=\mathfrak{b}|_U.
    $$
    Both $\mfa$ and $\mathfrak{b}$ are $T_N$-invariant. Therefore, the same equality holds on every translation $t \cdot U$ for $t \in T_N$.

    Now, it is enough to show that the translations of $U$ cover $X_\sigma$. Indeed, if $w \in \orb(\tau)$ for a face $\tau$ of $\sigma$, then
    $$
        \orb(\sigma) \subset \overline{\orb(\tau)} = \overline{T_N \cdot w}.
    $$
    Since $z \in U \cap \overline{T_N \cdot w}$, the orbit $T_N \cdot w$ meets $U$. Thus, $w$ belongs to some translate of $U$, and hence $X_\sigma=T_N \cdot U$. It follows that $\mfa=\mathfrak{b}$ on $X_\sigma$.
\end{proof}

Before proving a stronger result, we recall a standard topological fact about coherent analytic sheaves. We refer to \cite{bruhat1958lectures,grauert2004theory} for details in what follows. For every open subset $U \subset X_\sigma$, the spaces of sections of coherent analytic sheaves on $U$ carry their natural Fr\'echet topologies, and morphisms between coherent analytic sheaves induce continuous linear maps on their spaces of sections. In particular, the canonical projection from $\cO_{X_\sigma}$ onto $\cO_{X_\sigma}/\mfa$ induces a continuous linear map
$$
    \operatorname{pr} : \Gamma(U, \cO_{X_\sigma}) \longrightarrow \Gamma(U, \cO_{X_\sigma} / \mfa)
$$
with $\ker \operatorname{pr} = \Gamma (U,\mfa)$. Let $K_N \subset T_N$ be the maximal compact torus, equipped with its normalized Haar measure $dk$. If $F: K_N \to \cO_{X_\sigma}(U)$ is continuous, then its vector-valued integral is well-defined and
$$
    \operatorname{pr} \left( \int_{K_N} F(k) \, dk \right) = \int_{K_N} \operatorname{pr} ( F(k) ) \, dk.
$$

\begin{lemma}\label{lem:invariant_analytic_ideal_is_monomial}
    Let $\mfa$ be a coherent $T_N$-invariant analytic ideal sheaf on $X_\sigma$, and let $z\in\orb(\sigma)$. Then $\mfa_z$ is generated by finitely many germs of characters that are holomorphic near $z$.
\end{lemma}
\begin{proof}
    Since $X_\sigma$ is Stein, Cartan's theorem A implies that $\mfa_z$ is generated by the germs of global sections of $\mfa$. We will show that these germs belong to an ideal generated by monomial germs.

    For $f \in \Gamma(X_\sigma,\mfa)$ and $m \in M$, define
    $$
        P_m(f)(x) := \int_{K_N} \chi^{-m}(k) f(k \cdot x) \, dk.
    $$
    The composition $f \circ \tau_k$ depends continuously on $k \in K_N$ in the natural Fr\'echet topology. Since $\mfa$ is $T_N$-invariant, the continuity of the projection described above gives $P_m(f) \in \Gamma(X_\sigma,\mfa)$. Moreover, for all $k \in K_N$ and $x \in X_\sigma$, we have
    $$
        P_m(f)(k \cdot x) = \chi^m(k) P_m(f)(x).
    $$
    On $T_N$, the holomorphic function $P_m(f)/\chi^m$ is therefore $K_N$-invariant and hence constant. Consequently, every nonzero $P_m(f)$ is a constant multiple of a character $\chi^m$. Since $P_m(f)$ is holomorphic on $X_\sigma$, such an $m$ belongs to $\sigma^\vee \cap M$.

    Set $R := \cO_{X_\sigma,z}$, and let $\mathfrak{m}_z$ be its maximal ideal. Let $J \subset R$ be the ideal generated by the germs $(P_m(f))_z$ for all $f \in \Gamma(X_\sigma,\mfa)$ and $m \in M$. Then $J \subset \mfa_z$, and $J$ is generated by monomial germs. For each integer $\ell \ge 1$, consider the finite-dimensional vector space
    $$
        E_\ell := R/(J+\mathfrak{m}_z^\ell)
    $$
    and the natural map $q_\ell : \Gamma(X_\sigma,\cO_{X_\sigma}) \to E_\ell$. This map is continuous because it factors through the finite jet space $R/\mathfrak{m}_z^\ell$. For $f \in \Gamma(X_\sigma,\mfa)$, define the continuous function
    $$
        F_\ell : K_N \longrightarrow E_\ell,
        \qquad F_\ell(k) := q_\ell(f \circ \tau_k).
    $$
    Every Fourier coefficient of $F_\ell$ vanishes, since, for all $m \in M$,
    $$
        \int_{K_N} \chi^{-m}(k) F_\ell(k) \, dk
        = q_\ell(P_m(f)) = 0.
    $$
    Applying uniqueness of Fourier coefficients to each coordinate of $E_\ell$, we obtain $F_\ell=0$. Evaluating at the identity of $K_N$ yields $f_z \in J+\mathfrak{m}_z^\ell$. As this holds for every $\ell \ge 1$, Krull's intersection theorem applied to $R/J$ gives
    $$
        f_z \in \bigcap_{\ell \ge 1}(J+\mathfrak{m}_z^\ell)=J.
    $$
    Thus the germs of all global sections of $\mfa$ belong to $J$, and Cartan's theorem A gives $\mfa_z=J$. Finally, since $R$ is Noetherian, finitely many of the monomial germs defining $J$ generate $\mfa_z$.
\end{proof}

As a direct consequence, we have the following result that is stronger than Theorem~\ref{thm:local_to_global}.

\begin{corollary}\label{cor:global_monomial_generation}
    Every coherent $T_N$-invariant analytic ideal sheaf $\mfa \subset \cO_{X_\sigma}$ is globally generated by finitely many monomials. More precisely, there exist $m_1,\ldots,m_s \in \sigma^\vee \cap M$ such that
    $$
        \mfa = \chi^{m_1} \cO_{X_\sigma} + \cdots + \chi^{m_s} \cO_{X_\sigma}.
    $$
\end{corollary}

The same holds for coherent $T_N$-invariant $\cO_{X_\sigma}$-submodules of $\cM_{X_\sigma}$, with meromorphic monomials as generators.

\begin{corollary} \label{cor:global_meromorphic_monomial_generation}
    Every coherent $T_N$-invariant analytic $\cO_{X_\sigma}$-submodule $\cF \subset \cM_{X_\sigma}$ is globally generated by finitely many meromorphic monomials.
\end{corollary}
\begin{proof}
    Assume that $\cF$ is nonzero and denote by $\mathfrak{d}$ the sheaf of ideal quotients $(\cO_{X_\sigma}:\cF)$, which is defined by
    $$
        \mathfrak{d}(U) = \{ a \in \cO_{X_\sigma}(U) \,:\, a \cF|_U \subset \cO_U \}
    $$
    for every open subset $U \subset X_\sigma$.

    We first verify that $\mathfrak{d}$ is a coherent ideal sheaf with nonzero stalks. Since $\cF$ is coherent, it has locally finitely many meromorphic generators and we may write them as $g_1/h,\ldots, g_l/h$ for holomorphic functions $g_1,\ldots, g_l$ and a nonzero holomorphic function $h$. On such a neighborhood $U$, the sheaf $\mathfrak{d}|_U$ is the kernel of the morphism
    $$
        \cO_U \longrightarrow (\cO_U / h\cO_U)^{\oplus l}
    $$
    whose $i$-th component is multiplication by $g_i$ followed by the quotient map from $\cO_U$ to $\cO_U / h\cO_U$. Hence $\mathfrak{d}$ is coherent, and the inclusion $h \cO_U \subset \mathfrak{d}|_U$ shows that the stalks of $\mathfrak{d}$ are nonzero.

    The $T_N$-invariance of $\cF$ implies that $\mathfrak{d}$ is also $T_N$-invariant. By Corollary~\ref{cor:global_monomial_generation}, $\mathfrak{d}$ is globally generated by monomials. Choose a monomial generator $\chi^{m_0}$ of $\mathfrak{d}$, where $m_0 \in \sigma^\vee \cap M$. Then
    $$
        \chi^{m_0} \cF \subset \cO_{X_\sigma}
    $$
    is a coherent analytic ideal sheaf. Since $\chi^{m_0}\cF$ is $T_N$-invariant, by Corollary~\ref{cor:global_monomial_generation} again, we obtain $m_1,\ldots, m_s \in \sigma^\vee \cap M$ such that
    $$
        \chi^{m_0}\cF = \chi^{m_1}\cO_{X_\sigma} + \cdots + \chi^{m_s}\cO_{X_\sigma}.
    $$
    Consequently, $\cF$ is globally generated by the monomials $\chi^{m_1-m_0}, \ldots, \chi^{m_s-m_0}$.
\end{proof}

The following theorem shows that the $T_N$-invariance assumption in the preceding monomial generation results can be weakened to $K_N$-invariance.

\begin{theorem} \label{thm:compact_torus_invariant_imply_torus_invariant}
    Let $\cF \subset \cM_{X_\sigma}$ be a coherent analytic $\cO_{X_\sigma}$-submodule. If $\cF$ is $K_N$-invariant, then it is $T_N$-invariant.
\end{theorem}
\begin{proof}
    We first prove the assertion for coherent analytic ideal sheaves of $\cO_{X_\sigma}$. Let $\mfa$ be a coherent $K_N$-invariant analytic ideal sheaf of $\cO_{X_\sigma}$ and let $z$ be a point in $X_\sigma$. By the coherence of $\mfa$, we can choose a small neighborhood $W$ of $z$ and finitely many holomorphic functions $f_1,\ldots,f_s$ in $\Gamma(W,\mfa)$ that generate every stalk of $\mfa$ in $W$. Choose a relatively compact open neighborhood $V \Subset W$ of $z$.
    Fix $n \in N$ and choose a sufficiently small connected neighborhood $B \subset \bbC^*$ of $1$ such that $\gamma_n (\lambda) \cdot V \subset W$ for every $\lambda \in B$. Here, $\gamma_n:\bbC^* \to T_N$ is the \emph{one-parameter subgroup} associated with $n \in N$ defined by $\gamma_n(\lambda)(m) = \lambda^{\inp{m}{n}}$.

    Let $\operatorname{pr}_V: \Gamma(V,\cO_{X_\sigma}) \to \Gamma(V,\cO_{X_\sigma}/\mfa)$ be the natural projection map. For each $j$, define $F_j:B \to \Gamma(V,\cO_{X_\sigma}/\mfa)$ by
    $$
        F_j(\lambda) = \operatorname{pr}_V\bigl(f_j\circ\tau_{\gamma_n(\lambda)}\bigr).
    $$
    The holomorphicity of the action $\tau$ and the continuity of the linear map $\operatorname{pr}_V$ imply that $F_j$ is holomorphic with respect to the natural Fr\'echet topology on the space of sections.
    Indeed, since $X_\sigma$ is reduced, the natural Fr\'echet topology on $\Gamma(V,\cO_{X_\sigma})$ is the topology of uniform convergence on compact subsets. The function $G_j(\lambda,x) := f_j(\gamma_n(\lambda)\cdot x)$ is jointly holomorphic on $B \times V$, and Cauchy estimates in $\lambda$ show that its difference quotients converge uniformly on every compact subset of $V$ to its $\lambda$-derivative. Thus, the map assigning $G_j(\lambda, \cdot) \in \Gamma(V,\cO_{X_\sigma})$ to $\lambda$ is holomorphic in this Fr\'echet topology, and composing with the continuous linear map $\operatorname{pr}_V$ proves the holomorphicity of $F_j$.
    
    Since $\gamma_n(\lambda)\in K_N$ for $\lambda\in S^1$, the $K_N$-invariance of $\mfa$ gives $F_j(\lambda)=0$ for $\lambda\in B\cap S^1$. Applying the identity theorem after composition with continuous linear functionals, we obtain $F_j\equiv0$ on $B$. As $\ker\operatorname{pr}_V=\Gamma(V,\mfa)$, it follows that
    $$
        \bigl(\tau_{\gamma_n(\lambda)}^*\mfa\bigr)|_V\subset\mfa|_V
    $$
    for every $\lambda\in B$.
    After shrinking the neighborhoods, applying the same inclusion to the inverse translation gives equality. These local equalities extend along paths in $\bbC^*$, since the compact orbit path of each point can be covered by finitely many such neighborhoods. Thus $\mfa$ is invariant under $\gamma_n(\bbC^*)$. As $n\in N$ was arbitrary and the one-parameter subgroups generate $T_N$, the ideal sheaf $\mfa$ is $T_N$-invariant.

    Now, consider the general case. Set $\mathfrak{d}:=(\cO_{X_\sigma}:\cF)$. The local kernel description in the proof of Corollary~\ref{cor:global_meromorphic_monomial_generation} shows that $\mathfrak{d}$ is a coherent ideal sheaf with nonzero stalks. Since $\cF$ is $K_N$-invariant, so is $\mathfrak{d}$. Applying the ideal case proved above, we conclude that $\mathfrak{d}$ is $T_N$-invariant. By Corollary~\ref{cor:global_monomial_generation}, we can choose a monomial generator $\chi^{v_0}$ of $\mathfrak{d}$. Then
    $$
        \mathfrak{b} := \chi^{v_0} \cF \subset \cO_{X_\sigma}
    $$
    is a coherent analytic ideal sheaf. Since $\mathfrak{b}$ is $K_N$-invariant, by applying the ideal case once more, we conclude that $\mathfrak{b}$ is $T_N$-invariant. It follows that $\cF = \chi^{-v_0}\mathfrak{b}$ is $T_N$-invariant.
\end{proof}

We close this section with the weights $\varphi-c\varphi_\mfa$ from the introduction. Here $\varphi$ is a toric psh function on $X_\sigma$, and $D$ is a $T_N$-invariant Weil $\bbQ$-divisor on $X_\sigma$ such that $K_{X_\sigma}+D$ is $\bbQ$-Cartier.

\begin{proposition} \label{prop:toric_negative_ideal_weight}
    Let $\mfa \subset \cO_{X_\sigma}$ be a nonzero coherent $T_N$-invariant ideal sheaf and let $c$ be a positive rational number.
    Choose monomial generators
    $$
        \mfa= \chi^{a_1}\cO_{X_\sigma}+ \cdots + \chi^{a_s}\cO_{X_\sigma},
    $$
    and set $B_j := \divisor (\chi^{a_j})$. Then
    $$
        \cI[(X_\sigma,D);\varphi-c\varphi_\mfa] = \bigcap_{j=1}^s \cI[(X_\sigma,D-cB_j);\varphi].
    $$
    In particular, this sheaf is a coherent $T_N$-invariant $\cO_{X_\sigma}$-submodule of $\cM_{X_\sigma}$ globally generated by finitely many meromorphic monomials. Furthermore, with $c$ and $\mfa$ fixed,
    $$
        \cI[(X_\sigma,D);\varphi-c\varphi_\mfa] = \bigcup_{\delta>0} \cI[(X_\sigma,D);(1+\delta)\varphi-c\varphi_\mfa].
    $$
\end{proposition}
\begin{proof}
    The monomial generators exist by
    Corollary~\ref{cor:global_monomial_generation}.
    Up to locally bounded factors,
    $$
        e^{2c\varphi_\mfa} = \Big(\sum_{j=1}^s|\chi^{a_j}|^2\Big)^c \approx \sum_{j=1}^s|\chi^{a_j}|^{2c}.
    $$
    Each $B_j$ is principal Cartier, so $K_{X_\sigma}+D-cB_j$ is $\bbQ$-Cartier. If $dV_D$ and $dV_{D-cB_j}$ denote adapted volume forms for the corresponding pairs, then
    $$
        dV_{D-cB_j} \approx |\chi^{a_j}|^{2c}dV_D.
    $$
    Consequently, for a meromorphic germ $f$, the local integrability of $|f|^2e^{-2\varphi+2c\varphi_\mfa}dV_D$ is equivalent to that of $|f|^2e^{-2\varphi}dV_{D-cB_j}$ for every $j$. This proves the intersection formula.

    Each sheaf on the right is a coherent $\cO_{X_\sigma}$-submodule of $\cM_{X_\sigma}$ by Corollary~\ref{cor:multiplier_ideal_coherence} and is $T_N$-invariant by Theorem~\ref{thm:compact_torus_invariant_imply_torus_invariant}.
    Therefore, their finite intersection is also a coherent $T_N$-invariant $\cO_{X_\sigma}$-submodule of $\cM_{X_\sigma}$.
    Its finite meromorphic monomial generation follows from Corollary~\ref{cor:global_meromorphic_monomial_generation}.

    The last assertion follows from Corollary~\ref{cor:strong_openness} applied to the finitely many $(X_\sigma,D-cB_j)$, taking the minimum of the resulting positive parameters.
\end{proof}

\section{Multiplier ideals on a toric complex space} \label{sec:multiplier_ideals_on_a_toric_complex_space}

Let $\sigma$ be a strongly convex rational polyhedral cone in $N_\bbR$ and $D$ a $T_N$-invariant Weil $\bbQ$-divisor on $X_\sigma$ such that $K_{X_\sigma} + D$ is $\bbQ$-Cartier with index $q$. By Theorem~\ref{thm:toric_cartier_criterion}, there exists $u_D \in M_\bbQ$ such that
$$
    \inp{u_D}{v_\rho} = 1-a_\rho, \quad qu_D \in M
$$
for the primitive generators $v_\rho$ of the rays $\rho$ in $\sigma(1)$.

\begin{theorem}\label{thm:toric_multiplier_ideal}
    Let $\varphi$ be a toric psh function on $X_\sigma$. Then $\cI[(X_\sigma,D);\varphi]$ is a coherent $T_N$-invariant $\cO_{X_\sigma}$-submodule of $\cM_{X_\sigma}$ globally generated by finitely many meromorphic monomials.
    For every $m\in M$, the meromorphic monomial $\chi^m$ is a global section of this sheaf if and only if $m+u_D\in\interior P(\varphi)$.
\end{theorem}
\begin{proof}
    If $\sigma=\{0\}$, then $X_\sigma=T_N$ and $D=0$. The toric psh function $\varphi$ is locally bounded on $T_N$, so $\cI[(T_N,0);\varphi]=\cO_{T_N}$. Since $P(\varphi)=M_\bbR$, the conclusion follows. Henceforth assume that $\sigma \neq \{0\}$.

    Since $D$ and $\varphi$ are $K_N$-invariant, so is $\cI[(X_\sigma,D);\varphi]$. By Theorem~\ref{thm:compact_torus_invariant_imply_torus_invariant}, the sheaf $\cI[(X_\sigma,D);\varphi]$ is $T_N$-invariant. Thus, the first part is immediate from results in Section~\ref{sec:newton_convex_bodies_and_invariant_analytic_sheaves}.
    Just as in Corollary~\ref{cor:global_meromorphic_monomial_generation}, we can find a meromorphic monomial $\chi^{v_0}$ such that $\chi^{v_0}\cI[(X_\sigma,D);\varphi] \subset \cO_{X_\sigma}$. By Theorem~\ref{thm:local_to_global}, it suffices to show that, for a fixed $z \in \orb(\sigma)$, a meromorphic monomial $\chi^m$ belongs to $\cI[(X_\sigma,D);\varphi]_z$ if and only if $m + u_D \in \interior P(\varphi)$. Following the decompositions in (\ref{eqn:decomposition_of_N}) and (\ref{eqn:decompositions_of_N_and_M}), we replace $m \in M$ by $m+m'$ with $m \in M_\sigma$ and $m' \in M'$ for notational convenience. 
    
    Now, fix a point $z=(z_0,z')$ in $\orb(\sigma) = \{z_0\} \times T_{N'}$. A monomial $\chi^{m+m'}$ belongs to $\cI[(X_\sigma,D);\varphi]_z$ if and only if there exists a relatively compact $K_N$-invariant open neighborhood $U$ of $(z_0,z')$ such that
    \begin{equation} \label{ineq:L2_integral_of_monomial_is_finite}
        \int_{U}{|\chi^{m+m'}|^2e^{-2\varphi} dV_\beta} < \infty.
    \end{equation}
    Here $\beta$ can be taken to be a global generator of $\omega_{X_\sigma}^{[q]}(qD)$, since every $T_N$-invariant Cartier divisor on the affine toric variety $X_\sigma$ is principal.

    Let $m_1,\ldots, m_s$ be nonzero generators of $\sigma^\vee \cap M_\sigma$ and let $m'_1,\ldots, m'_l$ be generators of $M'$. Shrinking $U$ sufficiently, we may assume that $U$ is an open set given by
    $$
        U = \left\{ (w,w') \in X_{\sigma,N_\sigma} \times T_{N'} \,\middle|\, \substack{\displaystyle |\chi^{m_j}(w)|<\epsilon,\,j=1,\ldots,s \\[0.3em] \displaystyle c^{-1}|\chi^{m'_k}(z')| < |\chi^{m'_k}(w')| < c|\chi^{m'_k}(z')|, \, k=1,\ldots, l}\right\},
    $$
    where $\epsilon > 0$ is sufficiently small and $c > 1$ satisfies $1-\epsilon < c^{-1} < c < 1 + \epsilon$. We may restrict the domain of integration to $U \cap T_N$ since $U \backslash T_N$ is a closed subset of Lebesgue measure zero. Through the log map $\Log:T_N \to N_{\bbR}$, the set $U \cap T_N$ corresponds to the set of points $(x,x') \in N_{\sigma,\bbR} \oplus N_{\bbR}'$ satisfying
    \begin{gather*}
        \inp{m_j}{x} < \log{\epsilon} \quad \text{for all $j=1,\ldots,s$,} \\
        -\log{c} < \inp{m'_k}{x'} - \log{|\chi^{m'_k}(z')|} < \log{c} \quad \text{for all $k=1,\ldots,l$.}
    \end{gather*}
    Set
    $$
    \Omega' := \{ x' \in N'_\bbR \,:\, -\log{c} < \inp{m'_k}{x'} - \log{|\chi^{m'_k}(z')|} < \log{c}, \, k=1,\ldots,l\}.
    $$
    Choosing a vector $\mathbf{b}$ in the interior of $\sigma \subset N_{\sigma,\bbR}$ sufficiently far along an interior ray, and shrinking $U$ if necessary, we may replace the former set, considered as a subset of $N_{\sigma,\bbR}$, by $-\sigma - \mathbf{b}$. Therefore
    $$
        \Log (U \cap T_N) = (-\sigma - \mathbf{b}) \times \Omega'.
    $$
    Using the torus coordinates $t_1,\ldots,t_r$ given in Section~\ref{sec:newton_convex_bodies_and_invariant_analytic_sheaves}, we choose $\beta$ so that its adapted volume form $dV_\beta$ on $T_N$ can be written as
    $$
        dV_\beta = |t^{qu_D}|^{2/q} \left| \frac{dt_1 \wedge \cdots \wedge dt_r}{t_1 \cdots t_r} \right|^2.
    $$
    By a change of variables, we obtain
    \begin{equation} \label{eq:monomial_integral_to_polar_integral}
    \begin{aligned}
        \int_{U}{|\chi^{m+m'}|^2e^{-2\varphi} dV_\beta} 
        & =  \int_{U \cap T_N} |t^{q(m+m' + u_D)}|^{2/q} e^{-2\varphi} \left| \frac{dt_1 \wedge \cdots \wedge dt_r}{t_1 \cdots t_r} \right|^2 \\
        & = (4\pi)^r \int_{(-\sigma-\mathbf{b}) \times \Omega'} e^{2(\inp{m+m'+u_D}{(x,x')} - g(x,x'))} dxdx'.
    \end{aligned}
    \end{equation}
    The last integral is finite if and only if
    $$
        \inp{m+m'+u_D}{(x,0)} < (\rec g)(x,0) \quad \text{for every nonzero $x \in -\sigma$.}
    $$
    Indeed, necessity follows from Fubini's theorem, by applying Theorem~\ref{thm:negative_recession_iff_integrability} to a slice for some $x'$. Here, the recession function of $g(x-\mathbf{b},x')$ for a fixed $x'$ is equal to that of $g(x,0)$. Sufficiency follows from Remark~\ref{rem:uniform_coercivity}, since the linear exponential factor in $x'$ is bounded on a relatively compact open set $\Omega'$.
    Finally, by Proposition~\ref{prop:equivalent_def_of_newton_convex_body_and_other_characterization}, this condition is equivalent to $m+m'+u_D \in \interior P(\varphi)$, as desired.
\end{proof}

\begin{example} \label{ex:quadric_cone_with_boundary}
    Let $X=\{(z_1,z_2,z_3) \in \bbC^3 \,:\, z_1z_2=z_3^2\}$.  Let $E=\{(z_1,z_2,z_3) \in \bbC^3 \,:\, z_1=z_3=0 \}\subset X$, and set $D := \frac{1}{2} E$. Consider the toric psh function
    $$
        \varphi=\max\bigl\{ \sqrt{2}\log|z_1|,\,2\log|z_2| \bigr\}.
    $$
    This function does not have analytic singularities at the origin: its restrictions to the two coordinate axes have logarithmic coefficients $\sqrt{2}$ and $2$, whose ratio is irrational.

    The quotient presentation given by $z_1=s^2$, $z_2=t^2$, and $z_3=st$ identifies the character lattice with
    $$
        M = \{(a,b) \in \bbZ^2 \,:\, a+b \equiv 0 \pmod{2} \},
    $$
    where $z_1=\chi^{(2,0)}$, $z_2=\chi^{(0,2)}$, and $z_3=\chi^{(1,1)}$. The dual lattice is
    $$
        N = \bbZ(1,0) \oplus \bbZ\left(\frac{1}{2},\frac{1}{2}\right) \subset \bbR^2.
    $$
    With this lattice, the cone is $\sigma=\bbR_{\ge0}^2 \subset N_\bbR$, and $K_X+D$ is $\bbQ$-Cartier with $u_D=(\frac{1}{2},1)$. The associated convex function is
    $$
        g(x_1,x_2) = \max \bigl\{ 2\sqrt{2}\,x_1,\,4x_2 \bigr\},
    $$
    and thus
    $$
        P(\varphi) = \conv \bigl\{( 2\sqrt{2},0),(0,4) \bigr\}+\bbR_{\ge0}^2.
    $$
    By Theorem~\ref{thm:toric_multiplier_ideal}, a character $\chi^{(a,b)}$, with $(a,b)\in M$, is a global section of $\cI[(X,D);\varphi]$ if and only if
    $$
        a,b \ge 0, \quad \frac{a+\frac{1}{2}}{2\sqrt{2}}+\frac{b+1}{4} > 1.
    $$
    The minimal exponent vectors satisfying these conditions, with respect to addition by $\sigma^\vee \cap M$, are $(2,0)$, $(1,1)$, and $(0,4)$. Hence
    $$
        \cI[(X,D);\varphi] = (z_1,z_3,z_2^2) \, \cO_X.
    $$
    For comparison, $u_0 = (1,1)$, and the same calculation gives
    $$
        \cI[(X,0);\varphi] = (z_1,z_2,z_3) \, \cO_X.
    $$
    Thus the boundary changes the generator $z_2$ to $z_2^2$.
\end{example}

Combining Theorem~\ref{thm:toric_multiplier_ideal} with
Proposition~\ref{prop:toric_negative_ideal_weight}, we obtain
the following extension.

\begin{corollary}\label{cor:toric_negative_ideal_weight_criterion}
    Let $\varphi$ be a toric psh function on $X_\sigma$.
    Let $\mfa \subset \cO_{X_\sigma}$ be a nonzero coherent $T_N$-invariant ideal sheaf, and let $c$ be a positive rational number.
    Write
    $$
        \mfa = \chi^{a_1}\cO_{X_\sigma}+\cdots+\chi^{a_s}\cO_{X_\sigma}.
    $$
    Then $\cI[(X_\sigma,D);\varphi-c\varphi_\mfa]$ is a coherent $T_N$-invariant $\cO_{X_\sigma}$-submodule of $\cM_{X_\sigma}$ globally generated by finitely many meromorphic monomials.
    For every $m\in M$, the meromorphic monomial $\chi^m$ is a global section of this sheaf if and only if
    $m+u_D+ca_j\in\interior P(\varphi)$ for every $j=1,\ldots,s$.
\end{corollary}
\begin{proof}
    Set $B_j := \divisor(\chi^{a_j})$. By Proposition~\ref{prop:toric_negative_ideal_weight}, the sheaf is globally generated by finitely many meromorphic monomials in $\bbC[M]$ and
    $$
        \cI[(X_\sigma,D);\varphi-c\varphi_\mfa] = \bigcap_{j=1}^s \cI[(X_\sigma,D-cB_j);\varphi].
    $$
    Since $B_j$ has coefficient $\inp{a_j}{v_\rho}$ along $D_\rho$, we may take
    $$
        u_{D-cB_j} = u_D + ca_j.
    $$
    Applying Theorem~\ref{thm:toric_multiplier_ideal} to each pair $(X_\sigma,D-cB_j)$ gives the assertion.
\end{proof}

\begin{remark}
    A function $\varphi$ on $X_\sigma$ is called {\bf toric quasi-psh} if $\varphi = \psi+b$, where $\psi$ is a toric psh function on $X_\sigma$ and $b$ is a locally bounded measurable function.
    For such a function, define $P(\varphi):=P(\psi)$. This definition is independent of the choice of $\psi$:
    since $\Log^{-1}(-\sigma)$ is relatively compact in $X_\sigma$, the convex functions associated to any two such representatives differ by a bounded function on $-\sigma$.

    The conclusions of Theorem~\ref{thm:toric_multiplier_ideal} and Corollary~\ref{cor:toric_negative_ideal_weight_criterion} also hold for toric quasi-psh functions. Indeed, locally bounded perturbations do not change local integrability, so
    $$
        \cI[(X_\sigma,D);\varphi] = \cI[(X_\sigma,D);\psi].
    $$
    The same equality holds after subtracting $c\varphi_\mfa$ from both weights.
\end{remark}

As an application, we recover the following well-known facts, which appear in many contexts (cf. \cite[Proposition~3.7]{kollar1997singularities}).

\begin{corollary}
    Let $\Delta$ be a fan in $N_\bbR$ and let $X_\Delta$ be the toric complex analytic space associated to the fan $\Delta$. Denote by $D$ the reduced Weil divisor whose support is $X_\Delta \backslash T_N$. Then
    \begin{enumerate}
        \item $(X_\Delta,D)$ is lc, that is, $\cJ'(X_\Delta,D)=\cO_{X_\Delta}$.
        \item If $K_{X_\Delta}$ is $\bbQ$-Cartier, then $(X_\Delta,0)$ is klt, that is, $\cI[(X_\Delta,0);0] = \cO_{X_\Delta}$.
    \end{enumerate}
\end{corollary}
\begin{proof}
    Since both assertions are local, it suffices to show the corollary on each affine chart $X_\sigma$. Write $D_\sigma := D|_{X_\sigma}$.

    (1) Since $K_{X_\sigma}+D_\sigma=0$, we may take $u_{D_\sigma}=0$.
    Let
    $$
        \mathfrak{b} := \cI[(X_\sigma,D_\sigma);0].
    $$
    Theorem~\ref{thm:toric_multiplier_ideal} and $P(0) = \sigma^\vee$ show that $\mathfrak{b}$ is generated by monomials $\chi^{a_1},\ldots,\chi^{a_s}$ with $a_j \in \interior \sigma^\vee$ for every $j$. Thus, $\epsilon a_j \in \interior \sigma^\vee$ for every $\epsilon>0$. Corollary~\ref{cor:toric_negative_ideal_weight_criterion}, applied with $\varphi=0$, shows that
    $$
        1 \in \cI[(X_\sigma,D_\sigma); -\epsilon\varphi_{\mathfrak b}]
    $$
    for every rational $\epsilon>0$. Choose a sufficiently small rational $\epsilon>0$. By Proposition~\ref{prop:analytic_characterization_of_maximal_non-lc_ideal}, we obtain
    $$
        \cJ'(X_\sigma,D_\sigma) = \cO_{X_\sigma}.
    $$
    Hence $(X_\sigma,D_\sigma)$ is lc.

    (2) Assume that $K_{X_\sigma}$ is $\bbQ$-Cartier. We may choose $u_0 \in M_\bbQ$ such that 
    $$
        \inp{u_0}{v_\rho}=1
    $$
    for every ray $\rho$ of $\sigma$. Therefore, $u_0 \in \interior \sigma^\vee$. As $P(0)=\sigma^\vee$,
    Theorem~\ref{thm:toric_multiplier_ideal} gives
    $$
        1 \in \cI[(X_\sigma,0);0].
    $$
    This sheaf is an ideal sheaf of $\cO_{X_\sigma}$ by Proposition~\ref{prop:effective_imply_ideal}, so
    $$
        \cI[(X_\sigma,0);0] = \cO_{X_\sigma}.
    $$
    Equivalently, $(X_\sigma,0)$ is klt.
\end{proof}

\bibliographystyle{amsplain}
\bibliography{reference}

\end{document}